\documentclass[a4paper]{amsart}
\usepackage{amsthm,amsmath,amssymb}
\usepackage[latin1]{inputenc}
 \newtheorem{thm}{Theorem}
 \newtheorem{prop}[thm]{Proposition}
 \newtheorem{lemma}[thm]{Lemma}
 \newtheorem{kor}[thm]{Corollary}
 
 \theoremstyle{definition}
 \newtheorem{definition}[thm]{Definition}

 \theoremstyle{remark}
 \newtheorem{remark}[thm]{Remark}

\numberwithin{thm}{section}

 \def\id{{\rm id}}
 \def\Quot{{\rm Quot}}
\def\dom{{\rm dom}} 
\def\T{T}
\def\tr{tr}
\def\Tee{\mathcal{T}}
\def\bsh{$[b]$-short }
\def\sh{short }

\begin{document}
\begin{title}
{Truncations of level $1$ of elements in the loop group of a reductive group}
\end{title}
\author{Eva Viehmann}
\address{Technische Universit\"at M\"unchen\\Zentrum Mathematik - M11\\Boltzmannstr. 3\\85748 Garching\\Germany}
\date{}

\begin{abstract}{The aim of this article is to define and study a new invariant of elements of loop groups which is invariant under $\sigma$-conjugation by a hyperspecial maximal open subgroup and which we call the truncation of level 1. We classify truncations of level 1 and describe their specialization behavior. Furthermore we prove group-theoretic conditions for the set of $\sigma$-conjugacy classes obtained from elements of a given truncation of level 1 and in particular for the generic $\sigma$-conjugacy class in any given truncation stratum. In the last section we relate our invariant to the Ekedahl-Oort stratification of the Siegel moduli space and to generalizations to other PEL Shimura varieties.}
\end{abstract}
\maketitle
\section{Introduction}\label{intro}
Let $k$ be an algebraically closed field of characteristic $p>0$. Let $L$ be either $k((t))$ or $\Quot (W(k))$ and let $\mathcal{O}$ be the valuation ring. Here $W(k)$ is the ring of Witt vectors of $k$. We denote by $\sigma:x\mapsto x^{q}$ the Frobenius of $k$ over $\mathbb{F}_{q}$ for some fixed $q=p^r$ and also the Frobenius of $L$ over $F=\mathbb{F}_{q}((t))$ resp.~$F=\mathbb{Q}_{q}=\Quot (W(\mathbb{F}_q))$. Let $\mathcal{O}_F$ be the valuation ring of $F$. We denote the uniformizer $t$ or $p$ of $\mathcal{O}_F$ by $\epsilon$.

Let $G$ be a connected reductive group over $\mathcal{O}_F$. Then $G$ is  quasi-split and split over an unramified extension of $\mathcal{O}_F$ (compare Section \ref{secqsplit}). Let $B$ be a Borel subgroup of $G$ and let $\T$ be a maximal torus contained in $B$. Let $K=G(\mathcal{O})$ and let $I$ be the inverse image of $B(k)$ under the projection $K\rightarrow G(k)$. Let $K_1$ be the kernel of the projection $K\rightarrow G(k)$.

For $b\in G(L)$ we call $\{g^{-1}b\sigma(g)\mid g\in K\}$ the $K$-$\sigma$-conjugacy class of $b$, and $[b]=\{g^{-1}b\sigma(g)\mid g\in G(L)\}$ the $\sigma$-conjugacy class of $b$. In \cite{Kottwitz1} Kottwitz studies $\sigma$-conjugacy classes of elements of $G(L)$ and classifies them by two invariants, the Newton point and the Kottwitz point, cf. Section \ref{intro3}. In particular, he obtains a discrete invariant on the set of $K$-$\sigma$-conjugacy classes. The aim of this article is to study a second invariant of $K$-$\sigma$-conjugacy classes, namely the truncation of level 1 which we define as the associated $K$-$\sigma$-conjugacy class in $K_1\backslash G(L)/K_1$.

Note that for $L$ of mixed characteristic, both $K$-$\sigma$-conjugacy classes and $\sigma$-conjugacy classes occur naturally in the study of the reduction of Shimura varieties of PEL type, i.e.~for moduli spaces of abelian varieties or $p$-divisible groups with extra structure. For example, $p$-divisible groups of height $h$ over an algebraically closed field $k$ of characteristic $p$ are classified by their Dieudonn\'e modules. The Dieudonn\'e module is a pair $(\mathbf M,F)$ where $\mathbf M$ is a free $W(k)$-module of rank $h$ and where $F:\mathbf M\rightarrow \mathbf M$ is a $\sigma$-linear homomorphism satisfying $F(\mathbf M)\supseteq p\mathbf M$. Here $\sigma$ denotes the Frobenius of $W(k)$ over $\mathbb{Z}_p$. Choosing a basis for $\mathbf M$ we can write $F=b\sigma$ for some $b\in GL_h(W(k)[1/p])$. A change of the basis amounts to $\sigma$-conjugating $b$ by an element of $GL_h(W(k))=K$. Thus the isomorphism class of the $p$-divisible group corresponds to the $K$-$\sigma$-conjugacy class of $b$. Isogeny classes of $p$-divisible groups are likewise in bijection with rational Dieudonn\'e modules which are described by the $\sigma$-conjugacy classes of the corresponding elements $b\in GL_h(W(k)[1/p])$. In the function field case $L=k((t))$ a similar interpretation relates $K$-$\sigma$-conjugacy classes and conjugacy classes of elements of $G(L)$ to isomorphism classes and isogeny classes of local $G$-shtukas, respectively. 

\subsection{Classification of truncations of level 1}
 Let us first introduce some notation. Let $W=N_T(L)/T(L)$ denote the (absolute) Weyl group of $\T$ in $G$ where $N_T$ denotes the normalizer of $T$. Let $\widetilde{W}=N_T(L)/T(\mathcal{O})\cong W\ltimes X_*(\T)$ denote the extended affine Weyl group. For each $w\in W$ we choose a representative in $N_\T(\mathcal{O})$. We denote this representative by the same letter as the element itself. If $M$ is a Levi subgroup of $G$ containing $\T$ let $W_M$ be the Weyl group of $M$ and denote by $^{M}W$ resp.~${}^M\widetilde{W}$ the set of elements $x$ of $W$ resp.~$\widetilde{W}$ that are shortest representatives of their coset $W_Mx$. Similarly, $W^{M}$ denotes the set of elements $x$ that are the shortest representatives of their cosets $xW_{M}$ and accordingly for $\widetilde W$. For a dominant $\mu\in X_*(\T)$ let $M_{\mu}$ be the centralizer of $\mu$ and let $^{\mu}W=\sigma^{-1}( {}^{M_{\mu}}W)$. Let $P_{\mu}=M_{\mu}B$, a standard parabolic with Levi subgroup $M_{\mu}$. Let $x_{\mu}=w_0w_{0,{\mu}}$ where $w_0$ denotes the longest element of $W$ and where $w_{0,{\mu}}$ is the longest element of $W_{M_\mu}$. Let $\tau_{\mu}=x_{\mu}\epsilon^{\mu}$ where $\epsilon^{\mu}$ is the image of $\epsilon$ under $\mu:\mathbb{G}_m\rightarrow \T$. Then $\tau_{\mu}$ is the shortest element of $W\epsilon^{\mu}W$.

The classification of $K$-$\sigma$-conjugacy classes of elements of $K_1\backslash G(L)/K_1$ is given by the following theorem which we prove in Section \ref{sec2}. The second part of the theorem establishes a relation between the subdivisions of $K_1\backslash G(L)/K_1$ according to $K$-$\sigma$-conjugacy classes and according to Iwahori-double cosets.

\begin{thm}\label{propdef}
\begin{enumerate}
\item Let $\Tee=\{(w,\mu)\in W\times X_*(\T)_{\dom}\mid w\in {}^{\mu}W\}$. Then the map assigning to $(w,\mu)$ the $K$-$\sigma$-conjugacy class of $K_1w\tau_{\mu}K_1$ is a bijection between $\Tee$ and the set of $K$-$\sigma$-conjugacy classes in $K_1\backslash G(L)/K_1$.
\item Let $\mu\in X_*(T)_{\dom}$ and let $w\in {}^{\mu}W$. Then each element of $Iw\tau_{\mu}I$ is $I$-$\sigma$-conjugate to an element of $K_1w\tau_{\mu}K_1$.
\end{enumerate}
\end{thm}

\begin{definition}
We denote by $\tr$ the map $G(L)\rightarrow \Tee$ assigning to each $b$ the element of $\Tee$ corresponding to its $K$-$\sigma$-conjugacy class in $K_1\backslash G(L)/K_1$ under the bijection in Theorem \ref{propdef}. The pair $\tr(b)\in W\times X_*(\T)$ is called the \emph{truncation of level $1$} of $b$.
\end{definition}

Let $L=k((t))$. In this case we can also study the variation of the truncation of level $1$ in families. Let $LG$ be the loop group of $G_{\mathbb{F}_{q}}$, i.e.~the group ind-scheme representing the functor on $\mathbb{F}_{q}$-algebras $R\mapsto G(R((t)))$, compare \cite{Faltings}, Definition 1. We show in Section \ref{secclosure} that for each $(w,\mu)\in\Tee$ the set of $b\in G(L)$ with $\tr(b)=(w,\mu)$ is the set of $k$-valued points of a bounded locally closed subscheme of the loop group $LG$ of $G_{\mathbb{F}_{q}}$. For the notion of boundedness see Section \ref{secnot}.

\begin{definition}
Let $(w,\mu)\in\Tee$ and assume that char$(F)=p$. Let $S_{w,\mu}$ be the reduced subscheme of the loop group of $G_{\mathbb{F}_{q}}$ such that $S_{w,\mu}(k)$ consists of those $g\in G(k((t)))$ with $\tr(g)=(w,\mu)$.
\end{definition}

The closure of a stratum $S_{w,\mu}$ in $LG$ is a union of finitely many strata (see Lemma \ref{thmclosure}). 
\begin{thm}\label{thmclos} Let $S_{w',\mu'},S_{w,\mu}\subseteq LG$ be two truncation strata. Then 
$S_{w',\mu'}\subseteq\overline{S_{w,\mu}}$ if and only if there is a $\tilde{w}\in W$ with $\tilde{w}w'\tau_{\mu'}\sigma(\tilde{w})^{-1}\leq w\tau_{\mu}$ with respect to the Bruhat order.
\end{thm}

For $F=\mathbb{Q}_{q}$ it is not clear how to define an ind-scheme having $G(L)$ as its set of $k$-valued points. However one can study the stratifications induced on the reduction modulo $p$ of certain Shimura varieties. The main part of this paper is concerned with elements of $G(L)$ for both cases or, whenever a scheme structure is involved, the equicharacteristic case. The applications of our theory to Shimura varieties are detailed in Section \ref{secpolpdiv}.

\subsection{Truncations of level 1 and $\sigma$-conjugacy classes}\label{intro3}
A second major goal of this article is to compare the stratification of $LG$ by truncations of level 1 to the stratification by $\sigma$-conjugacy classes. More precisely, we study when a given truncation stratum intersects a given $\sigma$-conjugacy class non-trivially. Our main result in this context (Theorem \ref{thmmain}) is a necessary condition for non-emptiness of these intersections which determines in particular the generic $\sigma$-conjugacy class in each trunction stratum.

We first review Kottwitz's classification \cite{Kottwitz1} of the set $B(G)$ of $\sigma$-conjugacy classes of elements $b\in G(L)$ which generalizes the notion of Newton polygons (compare also \cite{RapoportRichartz}, Section 1 for a more complete review of these results). Each $\sigma$-conjugacy class is determined by two invariants. One of them is given by a map $\kappa_G:B(G)\rightarrow \pi_1(G)_{\Gamma}$ where $\pi_1(G)$ is the quotient of $X_*(\T)$ by the coroot lattice and where $\Gamma$ is the absolute Galois group of $F$. There is the following explicit description of $\kappa_G$. Let $b\in G(L)$ and let $\mu\in X_*(\T)$ be such that $b\in K\epsilon^{\mu}K$, compare Section \ref{sec23}. Then $\kappa_G(b)$ is the image of $\mu$ under the canonical projection from $X_*(\T)$ to $\pi_1(G)_{\Gamma}$. The second invariant is the so-called the Newton point $\nu=\nu_b$ of $b$, an element of $(X_*(T)_{\mathbb{Q}}/W)^{\Gamma}$, the set of $\Gamma$-invariant $W$-orbits on $X_*(T)\otimes \mathbb{Q}$. We usually consider the dominant representative of $\nu$, an element of $X_*(T)_{\mathbb{Q}}^{\Gamma}$ which we denote by the same letter $\nu$. This invariant is the direct analog of the usual Newton polygon classifying $F$-isocrystals over an algebraically closed field. The images of $\nu_b$ and $\kappa(b)$ in $\pi_1(G)_{\Gamma}\otimes \mathbb{Q}$ coincide. Note that Kottwitz's original article only considers the case of mixed characteristic, but the other case can be treated in exactly the same way. Furthermore, the two invariants $\nu$ and $\kappa$ lie in groups that are independent of the choice of $L$, compare Remark \ref{newremark'}. 

We further need the partial order on $B(G)$ defined by Rapoport and Richartz in \cite{RapoportRichartz}. It is given by $[b]\preceq [b']$ if and only if $\kappa_G(b)=\kappa_G(b')$ and $\nu_b\preceq\nu_{b'}$. Here the second condition means that $\tilde{\nu}_{b'}-\tilde{\nu}_b$ is a linear combination of positive coroots with coefficients in $\mathbb{Q}_{\geq 0}$ where $\tilde{\nu}_{b'}$ and $\tilde{\nu}_b$ are dominant representatives of the two orbits (compare Lemma 2.2 of loc.~cit.). Their Theorem 3.6 shows that for each $[b]$, the union of all $\sigma$-conjugacy classes which are less or equal to $[b]$ is closed in the loop group. More precisely, they show a corresponding statement over a field $F$ of mixed characteristic. The function field analog can be shown in a similar, but slightly easier way using properties of the affine Grassmannian, compare \cite{equidim}, Theorem 7.3. For split groups $G$, \cite{grothconj} shows that $\preceq$ describes the precise closure relations of the classes $[b]\subset LG$.

Let $[b]\in B(G)$. Let $M$ be the centralizer of the dominant Newton point $\nu_b$ of $b$, the Levi component of a standard parabolic subgroup defined over $\mathcal{O}_F$. In Section \ref{secfundalc} we define \bsh elements as elements $x$ of length 0 in $\widetilde{W}_M$ with $M$-dominant Newton point $\nu_b$ and $\kappa_G(x)=\kappa_G(b)$. In particular, \bsh elements are contained in $[b]$. The following theorem is a necessary condition for non-emptiness of intersections of truncation strata and $\sigma$-conjugacy classes. It is equivalent to non-emptiness of the intersection of a $\sigma$-conjugacy class with the closure of a truncation stratum and can (contrary to the definition of non-emptiness itself) be effectively checked in finite time.
\begin{thm}\label{thmmain}
Let $b\in G(k((t)))$, and let $(w,\mu)=\tr (b)$. Then there is a \bsh element $x\in \overline{S_{w,\mu}}$.
\end{thm}

We call an element $x\in\widetilde{W}$ \sh if it is \bsh for some $[b]\in B(G)$. Note that each stratum $S_{w,\mu}$ in the loop group is irreducible (Lemma \ref{thmclosure}), hence it contains a unique generic $\sigma$-conjugacy class. From Theorem \ref{thmmain}, one deduces the following corollary which characterizes this $\sigma$-conjugacy class. 

\begin{kor}\label{corgennp}
Let $[b]$ be the generic $\sigma$-conjugacy class in $S_{w,\mu}\subseteq LG$ for some $w\in {}^{\mu}W$. Then $[b]$ is equal to the unique maximal element in the set of $\sigma$-conjugacy classes of \sh elements $x\in \widetilde{W}$ such that $x\in  \overline{S_{w,\mu}}$. This is also the same as the maximal class $[x]$ among all $x\in \widetilde{W}$ with $x\leq w\tau_{\mu}$ in the Bruhat order. 
\end{kor}

\subsection{Comparison between equal and mixed characteristic}

In Section \ref{secfundalc} we prove the following theorem that allows to translate results between the function field case and the arithmetic case without having to repeat proofs. It uses that the set $B(G)$ of $\sigma$-conjugagcy classes of elements of $G(k((t)))$ can be canonically identified with that for $G(W(k)[1/p])$ using the invariants $\nu$ and $\kappa$ (Remark \ref{newremark'}).
\begin{thm}\label{corsuperset}
Let $(w,\mu)\in\Tee\subseteq W\times X_*(T)$. Then a $\sigma$-conjugacy class in $LG(k)$ contains an element of truncation type $(w,\mu)$ if and only if the corresponding $\sigma$-conjugacy class in $G(W(k)[1/p])$ contains an element of truncation type $(w,\mu)$.
\end{thm}

Using this comparison and Theorem \ref{thmclos} we obtain the following analog of Theorem \ref{thmmain} in the arithmetic context.
\begin{thm}
Let $(w,\mu)\in\Tee$ and let $b\in G(W(k)[1/p])$ with $\tr(b)=(w,\mu)$. Then there is a \bsh element $x$ satisfying the following condition. Let $\tr(x)=(w',\mu')$. Then there is a $\tilde{w}\in W$ with $\tilde{w}w'\tau_{\mu'}\sigma(\tilde{w})^{-1}\leq w\tau_{\mu}$.
\end{thm}

\subsection{Comparison with Ekedahl-Oort strata}\label{compeo}
Let $X$ be a $p$-divisible group over an algebraically closed field $k$ of characteristic $p$. Let $(\mathbf M,F)$ be its Dieudonn\'e module and write $F=b\sigma$ with $b\in GL_h(W(k)[1/p])$ with respect to some trivialization of $\mathbf M$. As $p\mathbf M\subseteq F(\mathbf M)\subseteq \mathbf M$, we have $b\in Kp^{\mu}K$ for some minuscule $\mu\in X_*(\T)$. 

In \cite{EO} Oort shows that one obtains a discrete invariant of $X$ (the so-called Ekedahl-Oort invariant) by considering the isomorphism class of the $p$-torsion points $X[p]$, or equivalently by studying the reduction modulo $p$ of the Dieudonn\'e module $\mathbf M$ together with the two maps induced by $F:\mathbf M\rightarrow \mathbf M$ and $V=pF^{-1}:\mathbf M\rightarrow \mathbf M$. Reformulating this invariant in terms of the element $b$, it corresponds to considering the $K_1$-double coset. In other words, we can apply our theory in the special case $G=GL_h$ and $\mu$ minuscule for $\mathcal{O}=W(k)$ to study the Ekedahl-Oort invariant of $p$-divisible groups. Likewise, truncations of level 1 for other groups yield classifications of Ekedahl-Oort invariants of $p$-divisible groups with extra structure by a polarization or endomorphisms. 

In Section \ref{secpolpdiv} we further study the relation between truncation strata in loop groups and Ekedahl-Oort strata in PEL Shimura varieties. Using Theorem \ref{corsuperset} we obtain a direct comparison for non-emptiness of intersections between truncation strata and $\sigma$-conjugacy classes on the one hand and between Ekedahl-Oort strata and isogeny classes of $p$-divisible groups on the other hand. It allows to deduce a non-emptiness criterion for Shimura varieties which is analogous to Theorem \ref{thmmain} and generalizes a result of Harashita which proved a conjecture by Oort.\\

\noindent{\it Acknowledgement.} I am grateful to Ulrich G\"ortz, Xuhua He, Torsten Wedhorn, and Chia-Fu Yu for helpful discussions. I thank Torsten Wedhorn for pointing out an error in a preliminary version of this article. I want to thank the referee for careful proofreading and several detailed remarks on how to improve the paper. Part of this paper was written during a stay at the Academia Sinica in Taipeh which was supported by the Academia Sinica and by the Hausdorff Research Center, Bonn. I thank the Academia Sinica for its hospitality, support and excellent working conditions. This work was partially supported by the SFB/TR45 ``Periods, Moduli Spaces and Arithmetic of Algebraic Varieties'' of the DFG.

\section{Reductive groups over local rings}\label{secnot}

In this section we summarize some facts about reductive groups over local rings that are used frequently in the paper.

\subsection{}\label{secqsplit} Let $G$ be a connected reductive group over $\mathcal{O}_F$. Then $G$ is quasi-split and split over an unramified extension of $\mathcal{O}_F$. Indeed, let $k_F$ be the residue field of $\mathcal{O}_F$. Then $G_{k_F}$ is quasi-split and split over a finite extension of $k_F$. Furthermore a Borel subgroup over $k_F$ and a split maximal torus over a finite extension of $k_F$ can be lifted to a Borel subgroup and a split maximal torus over $\mathcal{O}_F$ resp.~over the corresponding unramified extension of $\mathcal{O}_F$, compare \cite{VW}, A.4.

\subsection{}\label{sec22} The extended affine Weyl group $\widetilde{W}$ has a decomposition $\widetilde W\cong \Omega\ltimes W_{\rm aff}$. Here $\Omega$ is the subset of elements of $\widetilde{W}$ which fix the chosen Iwahori subgroup $I$ of $G(L)$. The second factor $W_{\rm aff}$ is the affine Weyl group of $G$. In terms of the decomposition $\widetilde{W}\cong W\ltimes X_*(\T)$ it has the following description. Let $G_{sc}$ be the simply connected cover of $G$ and let $\T_{sc}$ the inverse image of $\T$ in $G_{sc}$. Then $W_{\rm aff}\cong W\ltimes X_*(\T_{sc})$ and $\Omega\cong X_*(\T)/X_*(\T_{sc})$. The affine Weyl group of $G$ is an infinite Coxeter group. It is generated by the simple reflections $s_i$ associated with the simple roots of $T$ in $G$ together with the simple affine root. The choice of $I$ also induces an ordering on $\widetilde{W}$, the Bruhat ordering. It is defined as follows. Let $x,y\in\widetilde W$ and let $x=\omega_xx'$ and $y=\omega_yy'$ be their decompositions into elements of $\Omega$ and $W_{\rm aff}$. Then $x\leq y$ if and only if $\omega_x=\omega_y$ and if there are reduced expressions $x'=s_{i_1}\dotsm s_{i_n}$ and $y'=s_{j_1}\dotsm s_{j_m}$ for $x'$ and $y'$ such that $(s_{i_1},\dotsc,s_{i_n})$ is a subsequence of $(s_{j_1},\dotsc,s_{j_m})$.

Recall the morphism $\kappa_G:G(L)\rightarrow \pi_1(G)_{\Gamma}$ where $\pi_1(G)$ is the quotient of $X_*(T)$ by the coroot lattice. It induces a surjection $\kappa_G:\widetilde{W}\cong W\ltimes X_*(\T)\rightarrow \pi_1(G)_{\Gamma}$. The subgroup $W_{\rm aff}$ of $\widetilde{W}$ is in the kernel of $\kappa_G$. On the subgroup $\Omega$ of the extended affine Weyl group, $\kappa_G$ induces the canonical projection $\Omega\cong X_*(\T)/X_*(\T_{sc})\overset{\sim}{\rightarrow}\pi_1(G)\rightarrow \pi_1(G)_{\Gamma}$. 

\subsection{}\label{sec23} We have the following decompositions. More details can for example be found in \cite{42}.

{\it Iwasawa decomposition.} Let $P$ be a parabolic subgroup of $G$. Then $G(L)=P(L)K$.

{\it Bruhat-Tits decomposition.} $G(L)=\coprod_{x\in \widetilde W} IxI.$ In the function field case the double cosets are locally closed subschemes of $LG$. The closure of $IxI$ is equal to the union of all $Ix'I$ where $x'\leq x$ in the Bruhat order.

{\it Cartan decomposition.} $G(L)=\coprod_{\mu\in X_*(T)_{\dom}} K\epsilon^{\mu}K$ where $X_*(T)_{\dom}$ denotes the set of dominant elements of $X_*(T)$ and where $\epsilon^{\mu}$ is defined to be the image of $\epsilon$ under $\mu:\mathbb{G}_m\rightarrow T$. In the function field case the double cosets are locally closed subschemes of $LG$. The closure of $K\epsilon^{\mu}K$ is equal to the union of all $K\epsilon^{\mu'}K$ where $\mu'\preceq\mu$. Here $\mu'\preceq\mu$ if $\mu-\mu'$ is a non-negative integral linear combination of positive coroots.

{\it Iwahori decomposition.} Let $P$ be a standard parabolic subgroup of $G$ and let $N$ be its unipotent radical and $M$ the Levi factor containing $T$. Let $\overline{N}$ be the unipotent radical of the opposite parabolic. Let $I_{M}=I\cap M(L)$ and analogously for $N,\overline{N}$. Then $I=I_NI_MI_{\overline N}$.

\subsection{} A subset of the loop group $LG$ is called bounded if it is contained in a finite union of double cosets $K\epsilon^{\mu}K$. For $n\in \mathbb{N}$ let $K_n=\{g\in K\mid g\equiv 1\pmod{\epsilon^n}\}$. Then a subscheme $S$ of $LG$ is called admissible if there is an $n_S\in \mathbb{N}$ with $SK_{n_S}=S$. Let $S\subseteq LG$ be a bounded and admissible subscheme and let $n_S$ be as above. Let $\mathcal{B}$ be a finite union of double cosets containing $S$. Then $S$ can be studied by considering the image in $\mathcal{B}/K_{n_S}$ which is a scheme of finite type. For example $S$ is called locally closed if the same holds for its image in $\mathcal{B}/K_{n_S}$. The closure of $S$ is defined to be the inverse image under $LG\rightarrow LG/K_{n_S}$ of the closure of $S$ in $LG/K_{n_S}$. The subscheme $S$ is called smooth or irreducible if the same holds for its image in $\mathcal{B}/K_{n_S}$. Note that these notions do not depend on the choice of $\mathcal{B}$ and of $n_S$ provided that they are large enough.

\subsection{} The following lemma is a variant of the Theorem of Lang-Steinberg for the infinite-dimensional group schemes that we want to consider.
\begin{lemma}\label{thmlang}
Let $H\subseteq K$ be a subgroup of $K$. For all $n\in \mathbb{Z}_{\geq 0}$ let $H_n=\{h\in H\mid h\equiv 1\pmod{\epsilon^n}\}.$ We assume that $H/H_n$ and $H_{n-1}/H_n$ are connected linear algebraic groups for all $n$. Let $g\in G(L)$ with $g^{-1}H_ng\subseteq \sigma(H_n)$ for all $n$. Then the morphism $H\rightarrow H$ with $h\mapsto \sigma^{-1}(g^{-1}h^{-1}g)h$ is surjective. 
\end{lemma}
\begin{proof}
Let $h\in H$ and $n\in \mathbb{N}$. By the Theorem of Lang-Steinberg there is an $h_n\in H/H_n$ with $\sigma^{-1}(g^{-1}h_n^{-1}g)h_n\in hH_n$. We want to show that we can lift $h_n$ to an element $h_{n+1}\in H/H_{n+1}$ with $\sigma^{-1}(g^{-1}h_{n+1}^{-1}g)h_{n+1}\in hH_{n+1}$. Let $f_{n+1}\in H/H_{n+1}$ be an arbitrary lift of $h_n$. We now apply the Theorem of Lang-Steinberg to the morphism $H_{n}/H_{n+1}\rightarrow H_{n}/H_{n+1}$ with $\psi\mapsto h^{-1}\sigma^{-1}(g^{-1}\psi^{-1}f_{n+1}^{-1}g)f_{n+1}\psi$. Note that $H_{n+1}$ is a normal subgroup of $H$ for all $n$, hence this is indeed a well-defined element of $H_{n}/H_{n+1}$. Let $\psi_{n+1}$ be an inverse image of the identity element under this morphism. Then $h_{n+1}=f_{n+1}\psi_{n+1}$ is as claimed. Using induction and passing to the limit we obtain an element $h_{\infty}\in K$ with $\sigma^{-1}(g^{-1}h_{\infty}^{-1}g)h_{\infty}=h$. 
\end{proof}

\subsection{}\label{sec25} Let $P$ be a standard parabolic subgroup of $G$, i.e.~$B\subseteq P$. We denote by $M$ the Levi factor containing $T$ and by $N$ its unipotent radical. Let $\mu\in X_*(T)$. Let $\alpha$ be a root and $U_{\alpha}$ the corresponding root subgroup. Then $\epsilon^{\mu}U_{\alpha}(x)\epsilon^{-\mu}=U_{\alpha}(\epsilon^{\langle \alpha,\mu\rangle}x)$. In particular, we have $\epsilon^{\mu}N(\mathcal{O})\epsilon^{-\mu}\subseteq N(\mathcal{O})\cap K_1$ if $\langle \alpha,\mu\rangle>0$ for all roots $\alpha$ of $T$ in $N$. This is for example the case if $\mu$ is dominant and $M$ contains the centralizer of $\mu$.

\section{Truncations of level $1$}\label{sec2}
The goal of this section is to prove Theorem \ref{propdef}, in particular we allow both the function field case and the case of mixed characteristic. The proof follows a strategy by B\'edard, \cite{bedard}.

\begin{proof}[Proof of Theorem \ref{propdef}]
Let $b\in G(L)$. By the Cartan decomposition there is a unique dominant $\mu\in X_*(\T)$ with $b\in K\epsilon^{\mu}K$ and $b$ is $K$-$\sigma$-conjugate to an element of the form $b_0x_{\mu}\epsilon^{\mu}=b_0\tau_{\mu}$ with $b_0\in K$. 

To show (1) we have to prove that there is a unique $w\in {}^{\mu}W$ such that $b_0\tau_{\mu}$ is $K$-$\sigma$-conjugate to an element of $K_1w\tau_{\mu}K_1$. We use induction on $i$ to show that there exist a sequence of elements $u_i\in W$, two sequences of standard Levi subgroups $M_i, M_i'$ of $G$, and a sequence of elements $b_i\in M_i'(\mathcal{O})$ with the following properties:
\begin{enumerate} 
\item[a)] $M_0=M_0'=G$,\\ $M_1'=x_{\mu}M_{\mu}x_{\mu}^{-1}$ and $M_1=\sigma^{-1}(M_{\mu})$,\\ $M_i=M_{i-1}'\cap u_{i-1}^{-1}\sigma^{-1}(M_{\mu})u_{i-1}$, and\\ $M_i'=M_1'\cap   x_{\mu}\sigma(u_{i-1}M'_{i-1}u_{i-1}^{-1})x_{\mu}^{-1}=x_{\mu}\sigma(u_{i-1}M_iu_{i-1}^{-1})x_{\mu}^{-1}$ for $i>1$.
\item[b)] $u_0=1$ and \\$u_i=u_{i-1}\delta_i$ for $i>0$ for some $\delta_i\in W_{M'_{i-1}}$ which is the shortest representative of $W_{M_i}\delta_iW_{M'_{i}}$ and\\
$u_i$ is the shortest representative of $W_{M_1}u_iW_{M'_{i}}$.
\item[c)] $b$ is $K$-$\sigma$-conjugate to an element of $K_1u_ib_i\tau_{\mu}K_1$.
\item[d)] $u_i W_{M'_i}\subseteq W$ is uniquely determined by the $K$-$\sigma$-conjugacy class of $b$ in $K_1\backslash G(L)/K_1$.
\item[e)] $u_ib'_i\tau_{\mu}$ with $b'_i\in M_i'(\mathcal{O})$ is in the $K$-$\sigma$-conjugacy class of $b$ in $K_1\backslash G(L)/K_1$ if and only if the images in $G(k)=G(\mathcal{O})/K_1$ of the two elements $b_i,b_i'$ are in the same $M_{i+1}(k)$-orbit in $U_{P_{i+1}}(k)\backslash M_i'(k)/U_{\overline{P}_{i+1}'}(k)$ under the action 
\begin{eqnarray*}
M_{i+1}(\mathcal{O})\times U_{P_{i+1}}(k)\backslash M_i'(k)/U_{\overline{P}_{i+1}'}(k)&\rightarrow &U_{P_{i+1}}(k)\backslash M_i'(k)/U_{\overline{P}_{i+1}'}(k)\\
(g,m)&\mapsto &g^{-1}mx_{\mu}\sigma(u_igu_i^{-1})x_{\mu}^{-1}.
\end{eqnarray*}
Here\\ $P_0=P'_0=G$,\\ $P_1'=x_{\mu}P_{\mu}x_{\mu}^{-1}$ and $P_1=\sigma^{-1}(P_{\mu})$,\\ $P_{i}=M'_{i-1}\cap u_{i-1}^{-1}\sigma^{-1}(P_{\mu})u_{i-1}$ and\\ $P'_{i}=M_1'\cap x_{\mu}\sigma(u_{i-1}P'_{i-1}u_{i-1}^{-1})x_{\mu}^{-1}$ for $i>1$, are parabolic subgroups of $M'_{i-1}$.\\ Furthermore, $U_P$ denotes the unipotent radical of a linear algebraic group $P$.
\end{enumerate}
Before we begin the proof, let us show that some of the conditions automatically follow from the others. We first check inductively that a) and b) imply that $M_i'\supseteq P'_{i+1}$ and that $M'_{i+1}$ is the Levi subgroup of $P'_{i+1}$ containing $T$, for all $i$.  For $i=0$ this is obvious. The second statement is also clear by induction. From b) and the induction hypothesis $P_i'\subseteq M_{i-1}'$ we see that $P_{i+1}'=M_1'\cap x_{\mu}\sigma(u_iP_i'u_i^{-1})x_{\mu}^{-1}$ is contained in $M_1'\cap x_{\mu}\sigma(u_{i-1}M_{i-1}'u_{i-1}^{-1})x_{\mu}^{-1}=M_i'$.

Results of B\'edard \cite{bedard}, or Lusztig (\cite{Lusztig}, (a)--(d) in the proof of Proposition 2.4) show that the last condition in b) follows automatically using induction, using the condition on $\delta_i$ and the definition of $M_i$ and $M_i'$.

Note that if $g\in M_{i+1}(\mathcal{O})$ then $x_{\mu}\sigma(u_igu_i^{-1})x_{\mu}^{-1}\in M_{i+1}'(\mathcal{O})$, so the action in e) is well-defined.

{\it Claim. }Conditions (a) and (b) above imply that the $M_i$ and $M_i'$ are standard Levi subgroups. 

We show this claim using induction. The Levi subgroups $M_{\mu}$ and $\sigma^{-1}(M_{\mu})$ are standard as $\mu$ is dominant and as $B$ is invariant under $\sigma$. Now we use the following fact: If $M$ is a standard Levi, if $\alpha$ is a simple root which lies in $M$ and if $x\in W^M$ then $x(\alpha)$ is again positive. If each such $x(\alpha)$ is again simple then $xMx^{-1}$ is again a standard Levi subgroup. For $x=x_{\mu}=w_0w_{0,\mu}$ this implies that $x_{\mu}M_{\mu}x_{\mu}^{-1}$ is standard. For the induction step we show that if $M,M'$ are standard and $x$ is the shortest representative of $W_{M'}xW_M$, then $M'\cap xMx^{-1}$ is also standard. By the above fact it is enough to show that for every simple root $\alpha$ in $M$ such that $x(\alpha)$ is a root in $M'$, this root is also simple. We have $1+\ell(x)=\ell(xs_{\alpha})=\ell(xs_{\alpha}x^{-1})+\ell(x)$ where the last equality uses that $xs_{\alpha}x^{-1}\in W_{M'}$, compare \cite{DDPW}, Lemma 4.17. Hence $\ell(xs_{\alpha}x^{-1})=1$, and $x(\alpha)$ is simple. We apply this first to $u_{i-1}^{-1}\in W^{M_1}\cap {}^{M_{i-1}'}W$, $M_1$ and $M_{i-1}'$ to obtain inductively that $M_i=M_{i-1}'\cap u_{i-1}^{-1}M_1u_{i-1}$ is standard. For $M_i'$ the properties of $\sigma$ and $x_{\mu}$ already used above imply that it is enough to show that $M_1\cap u_{i-1}M'_{i-1}u_{i-1}^{-1}$ is standard. This follows in the same way as before. 

We now carry out the induction to show a)--e). For $i=0$, d) is obvious and c) has been shown above. For e) Section \ref{sec25} implies that on $K_1$-double cosets, the effect of $\sigma$-conjugation of $b_0\tau_{\mu}$ by $g\in U_{P_1}(\mathcal{O})$ is the same as left multiplication of $b_0$ by $g$. Similarly one sees $\epsilon^{-\mu}U_{\overline{P}_{\mu}}(\mathcal{O})\epsilon^{\mu}\subset K_1$. Hence right multiplication of $b_0$ by an element of $U_{\overline{P}_{1}'}(\mathcal{O})$ does not change the class $b_0\tau_{\mu}K_1$.  The effect of the action of $M_1(\mathcal{O})$ on $b_0$ corresponds to $\sigma$-conjugation of $b_0\tau_{\mu}$ and thus it does not change the $K$-$\sigma$-conjugacy class of $b_0\tau_{\mu}$. For the other direction, if an element $g\in K$ conjugates $b_0\tau_{\mu}$ into $M_0'(\mathcal{O})\tau_{\mu}=K\epsilon^{\mu}$, then $\sigma(g)\in K\cap \epsilon^{-\mu}K\epsilon^{\mu}$. In particular $g$ is contained in the parahoric subgroup of $K$ of elements whose image in $G(k)$ is in $P_1(k)$. Using the analog of the Iwahori decomposition for this subgroup and the fact that $\sigma(g)\in \epsilon^{-\mu}K\epsilon^{\mu}$ we obtain a decomposition of $g$ into factors in $U_{P_1}(\mathcal{O})$, $M_1(\mathcal{O})$ and $\epsilon^{-\sigma^{-1}(\mu)}U_{\overline{P}_1}(\mathcal{O})\epsilon^{\sigma^{-1}(\mu)}=\sigma^{-1}(\epsilon^{-\mu}U_{\overline{P}_{\mu}}(\mathcal{O})\epsilon^{\mu})\subset K_1$. This shows e) and finishes the argument for $i=0$.

We have to show that a)--e) for some $i$ imply the same properties for $i+1$. Let $b_i$ be as in c). We decompose $b_i$ using the Bruhat decomposition to obtain that $b_i\in K_1P_{i+1}(\mathcal{O})\delta_{i+1}\overline{P}'_{i+1}(\mathcal{O})$ for some $\delta_{i+1}$ as in b) and with $P_{i+1}$ and ${P}_{i+1}'$ as in e). We may assume that the factor in $K_1$ is trivial. By e) we may further assume that the factors in $P_{i+1}(\mathcal{O})$ and $\overline{P}_{i+1}'(\mathcal{O})$ lie in $M_{i+1}(\mathcal{O})$ and $M_{i+1}'(\mathcal{O})$, respectively. We obtain a decomposition $u_ib_i\tau_{\mu}\in \bigl(u_iM_{i+1}(\mathcal{O})u_i^{-1}\bigr)u_i\delta_{i+1}M'_{i+1}(\mathcal{O})\tau_{\mu}$. After $\sigma$-conjugating $u_ib_i\tau_{\mu}$ with the factor in $u_iM_{i+1}(\mathcal{O})u_i^{-1}$ and using that $\sigma(u_iM_{i+1}(\mathcal{O})u_i^{-1})=x_{\mu}^{-1}M'_{i+1}x_{\mu}\subseteq x_{\mu}^{-1}M'_1x_{\mu}= M_{\mu}$ we obtain c) for $i+1$. Property d) follows from the uniqueness of the Bruhat decomposition together with d) and e) for $i$. It remains to show e). If we replace $b_{i+1}$ by $\beta b_{i+1}$ where $\beta$ is an element of $U_{P_{i+2}}(\mathcal{O})$ this has the effect that the product $u_{i+1}b_{i+1}\tau_{\mu}$ is multiplied on the left with an element $\delta(\beta)$ of $u_{i+1}U_{P_{i+2}}(\mathcal{O})u_{i+1}^{-1}=U_{P_1}(\mathcal{O})\cap u_{i+1}M_{i+1}'(\mathcal{O})u_{i+1}^{-1}$. This does not change the $K$-$\sigma$-conjugacy class in $K_1\backslash G(L)/K_1$. Indeed, by Section \ref{sec25} right multiplication by elements of $\sigma(U_{P_1}(\mathcal{O})\cap u_{i+1}M_{i+1}'(\mathcal{O})u_{i+1}^{-1})$ does not change the coset $K_1u_{i+1}b_{i+1}\tau_{\mu}$, and hence $K_1\delta(\beta)u_{i+1}b_{i+1}\tau_{\mu}=K_1\delta(\beta)u_{i+1}b_{i+1}\tau_{\mu}\sigma(\delta(\beta)^{-1})$. Now we want to show that replacing $b_{i+1}$ by $b_{i+1}\beta$ with $\beta \in U_{\overline{P}'_{i+2}}(\mathcal{O})$ also does not change the $K$-$\sigma$-conjugacy class of $u_{i+1}b_{i+1}\tau_{\mu}$. As $M_{i+1}'\subseteq x_{\mu}M_{\mu}x_{\mu}^{-1}$ this replacement has the effect that $u_{i+1}b_{i+1}\tau_{\mu}=u_{i+1}b_{i+1}x_{\mu}\epsilon^{\mu}$ is multiplied on the right with $\tau_{\mu}^{-1}\beta\tau_{\mu}=x_{\mu}^{-1}\beta x_{\mu}$. We $\sigma$-conjugate with the element $\sigma^{-1}(x_{\mu}^{-1}\beta x_{\mu})^{-1}\in \sigma^{-1}(x_{\mu}^{-1}U_{\overline{P}'_{i+2}}(\mathcal{O}) x_{\mu})=\sigma^{-1}(M_{\mu}(\mathcal{O}))\cap u_{i+1}U_{\overline{P}'_{i+1}}(\mathcal{O})u_{i+1}^{-1}$ (which is in particular in $K$). Then we obtain an element of $u_{i+1}U_{\overline{P}_{i+1}'}(\mathcal{O})b_{i+1}x_{\mu}\epsilon^{\mu}$. As $b_{i+1}\in M'_{i+1}(\mathcal{O})$, the element lies in $u_{i+1}b_{i+1}U_{\overline{P}_{i+1}'}(\mathcal{O})x_{\mu}\epsilon^{\mu}$. Using induction we obtain that this is contained in the same class as $u_{i+1}b_{i+1}x_{\mu}\epsilon^{\mu}$. Finally, the effect of the action of $M_{i+2}(\mathcal{O})$ on $b_{i+1}$ corresponds to $\sigma$-conjugation of $u_{i+1}b_{i+1}\tau_{\mu}$ by elements of $u_{i+1}M_{i+2}(\mathcal{O})u_{i+1}^{-1}$, and thus it leaves the $K$-$\sigma$-conjugacy class of $u_{i+1}b_{i+1}\tau_{\mu}$ stable.
It remains to show the other direction of e). So assume that $u_{i+1}b'_{i+1}\tau_{\mu}=u_{i}(\delta_{{i+1}}b'_{i+1})\tau_{\mu}$ with $b'_{i+1}\in M_{i+1}'(\mathcal{O})$ is in the $K$-$\sigma$-conjugacy class of $u_{i+1}b_{i+1}\tau_{\mu}=u_{i}(\delta_{{i+1}}b_{i+1})\tau_{\mu}$ in $K_1\backslash G(L)/K_1$. Using induction for $\delta_{{i+1}}b'_{i+1}, \delta_{{i+1}}b_{i+1}\in M_i'(\mathcal{O})$ we obtain elements $g\in M_{{i+1}}(\mathcal{O})$, $a\in U_{P_{i+1}}(\mathcal{O})$ and $a'\in U_{\overline{P}_{i+1}'}(\mathcal{O})$ with 
$$g^{-1}\delta_{i+1}b_{i+1}x_{\mu}\sigma(u_{i}gu_{i}^{-1})x_{\mu}^{-1}=a\delta_{i+1}b_{i+1}'a'.$$
Let $h=\delta_{i+1}^{-1}g\delta_{i+1}$ and $\tilde a=b_{i+1}'a'(b_{i+1}')^{-1}$. Then $\tilde a \in  U_{\overline{P}_{i+1}'}$, and
\begin{equation}\label{eqnew1}h^{-1}b_{i+1}x_{\mu}\sigma(u_{{i+1}}hu_{{i+1}}^{-1})x_{\mu}^{-1}=\delta_{i+1}^{-1}a\delta_{i+1}\tilde{a}b_{i+1}'.\end{equation}
 Notice that if $P,Q$ are connected linear algebraic subgroups containing $T$ such that $P$ is parabolic and if we denote $P=MU_P$ the decomposition into the Levi subgroup containing $T$ and the unipotent radical, then 
\begin{equation}\label{star}P\cap Q=(U_P\cap Q)(M\cap Q).\end{equation} Indeed, both sides contain $T$ and the same root subgroups, and are generated by these subgroups.

We have $b_{i+1},b_{i+1}', x_{\mu}\sigma(u_{{i+1}}hu_{{i+1}}^{-1})x_{\mu}^{-1}\in M_{i+1}'(\mathcal{O})$. Thus (\ref{eqnew1}) implies that 
\begin{equation}\label{eqnew25}
h\delta_{i+1}^{-1}a\delta_{i+1}\tilde{a}\in M_{i+1}'(\mathcal{O}).
\end{equation}
Hence  $h\delta_{i+1}^{-1}a\delta_{i+1}\in \overline P_{i+1}'(\mathcal{O}),$ and (as it is equal to $\delta_{i+1}^{-1}ga\delta_{i+1}$), it is also contained in $(\delta_{i+1}^{-1}P_{i+1}\delta_{i+1})(\mathcal{O})$.
Using this and (\ref{star}) we obtain that 
\begin{eqnarray*}
\tilde a&\in& (U_{\overline{P}_{i+1}'}\cap (M_{i+1}'\cdot(\overline{P}_{i+1}'\cap\delta_{i+1}^{-1}P_{i+1}\delta_{i+1})))(\mathcal{O})\\
&=&(U_{\overline{P}_{i+1}'}\cap (M_{i+1}'\cdot(U_{\overline{P}_{i+1}'}\cap\delta_{i+1}^{-1}P_{i+1}\delta_{i+1})))(\mathcal{O})\\
&=&(U_{\overline{P}_{i+1}'}\cap \delta_{i+1}^{-1}P_{i+1}\delta_{i+1})(\mathcal{O})\\
&=&\big((U_{\overline{P}_{i+1}'}\cap \delta_{i+1}^{-1}U_{P_{i+1}}\delta_{i+1})(U_{\overline{P}_{i+1}'}\cap \delta_{i+1}^{-1}M_{i+1}\delta_{i+1})\big)(\mathcal{O}).
\end{eqnarray*}
We write $\tilde a=a_1a_2$ for this decomposition. Replacing $\tilde a$ by $a_2$ and $a$ by $a\delta_{i+1}a_1\delta_{i+1}^{-1}\in U_{P_{i+1}}(\mathcal{O})$ the right hand side of (\ref{eqnew1}) does not change. Hence we may assume that $\tilde a\in (U_{\overline{P}_{i+1}'}\cap \delta_{i+1}^{-1}M_{i+1}\delta_{i+1})(\mathcal{O}).$ 

In particular, as now $\tilde a\in (\delta_{i+1}^{-1}M_{i+1}\delta_{i+1})(\mathcal{O})$ we have $\tilde a^{-1}\delta_{i+1}^{-1}a\delta_{i+1}\tilde{a}\in (\delta_{i+1}^{-1}U_{P_{i+1}}\delta_{i+1})(\mathcal{O})$. Equation (\ref{eqnew1}) is equivalent to 
$$\tilde a (h\tilde a )^{-1}b_{i+1}x_{\mu}\sigma(u_{i+1}h\tilde au_{i+1}^{-1})x_{\mu}^{-1}(x_{\mu}\sigma(u_{i+1}\tilde a^{-1}u_{i+1}^{-1})x_{\mu}^{-1})=\tilde a(\tilde a^{-1}\delta_{i+1}^{-1}a\delta_{i+1}\tilde{a})b_{i+1}'$$
Replacing $h$ by $h\tilde a \in (\delta_{i+1}^{-1}M_{i+1}\delta_{i+1})(\mathcal{O})$ and $\delta_{i+1}^{-1}a\delta_{i+1}$ by $\tilde a^{-1}\delta_{i+1}^{-1}a\delta_{i+1}\tilde{a}$ we obtain the equivalent equation (using these new variables)
$$h^{-1}b_{i+1}x_{\mu}\sigma(u_{{i+1}}hu_{{i+1}}^{-1})x_{\mu}^{-1}=\delta_{i+1}^{-1}a\delta_{i+1}b_{i+1}'\zeta$$
where $$\zeta=x_{\mu}\sigma(u_{i+1}\tilde a u_{i+1}^{-1})x_{\mu}^{-1}\in (M_{i+1}'\cap x_{\mu}\sigma(u_{i+1}U_{\overline{P}_{i+1}'} u_{i+1}^{-1})x_{\mu}^{-1})(\mathcal{O})=U_{\overline{P}_{i+2}'}(\mathcal{O}).$$
As we may multiply $b_{i+1}'$ on the right by elements in $U_{\overline{P}_{i+2}'}(\mathcal{O})$ we may assume that $\zeta=1$, which corresponds to (\ref{eqnew1}) for $\tilde a=1$. In particular, (\ref{eqnew25}) yields $h\delta_{i+1}^{-1}a\delta_{i+1}\in M_{i+1}'(\mathcal{O}),$ and as before it is also an element of $(\delta_{i+1}^{-1}P_{i+1}\delta_{i+1})(\mathcal{O})$. Thus by (\ref{star}) we have $$h\delta_{i+1}^{-1}a\delta_{i+1}\in ((\delta_{i+1}^{-1}M_{i+1}\delta_{i+1}\cap M_{i+1}')(\delta_{i+1}^{-1}U_{P_{i+1}}\delta_{i+1}\cap M_{i+1}'))(\mathcal{O}).$$ As $h\in \delta_{i+1}^{-1}M_{i+1}\delta_{i+1}$ and $a\in U_{P_{i+1}}$ this implies that
\begin{eqnarray}
\nonumber h&\in& (\delta_{i+1}^{-1}M_{i+1}\delta_{i+1}\cap M_{i+1}')(\mathcal{O})\\
\label{heute2} \delta_{i+1}^{-1}a\delta_{i+1}&\in&(\delta_{i+1}^{-1}U_{P_{i+1}}\delta_{i+1}\cap M_{i+1}')(\mathcal{O}).
\end{eqnarray}
We obtain
\begin{eqnarray*}
 h&\in& (\delta_{i+1}^{-1}M_{i+1}\delta_{i+1}\cap M_{i+1}')(\mathcal{O})\\
&=&(u_{i+1}^{-1}\sigma^{-1}(x_{\mu}^{-1}M_{i+1}'x_{\mu})u_{i+1}\cap M_{i+1}')(\mathcal{O})\\
&\subseteq&(u_{i+1}^{-1}\sigma^{-1}(x_{\mu}^{-1}M_{1}'x_{\mu})u_{i+1}\cap M_{i+1}')(\mathcal{O})\\
&=&(u_{i+1}^{-1}M_1u_{i+1}\cap M_{i+1}')(\mathcal{O})=M_{i+2}(\mathcal{O}).
\end{eqnarray*}
By definition, $u_{i+1}\delta_{i+1}^{-1}U_{P_{i+1}}\delta_{i+1}u_{i+1}^{-1}=u_{i}U_{P_{i+1}}u_{i}^{-1}\subseteq U_{P_1}$. Thus (\ref{heute2}) implies
$$\delta_{i+1}^{-1}a\delta_{i+1}\in(u_{i+1}^{-1}U_{P_{1}}u_{i+1}\cap M_{i+1}'))(\mathcal{O})=U_{P_{i+2}}(\mathcal{O}).$$ Altogether this means that via the elements $h\in M_{i+2}(\mathcal{O})$, $\delta_{i+1}^{-1}a\delta_{i+1}\in U_{P_{i+2}}(\mathcal{O})$, and $\zeta\in U_{\overline{P}_{i+2}'}(\mathcal{O})$, we proved that the two elements $b_{i+1},b_{i+1}'\in M_{i+1}'(\mathcal{O})$ are in the same $M_{i+2}(\mathcal{O})$-orbit in $U_{P_{i+2}}(k)\backslash M_{i+1}'(k)/U_{\overline{P}_{i+2}'}(k)$. This finishes the induction step for e) and completes the induction.

The $M_i'$ form a decreasing family of Levi subgroups and thus become constant after finitely many steps. Thus for $n$ sufficiently large,  $x_{\mu}\sigma(u_nM'_nu_n^{-1})x_{\mu}^{-1}=M'_n=M_{n+1}'$, and $M_n=M_n'$. As $M'_n=M_{n+1}'$ we obtain $\sigma(u_nM_n'u_n)\subseteq x_{\mu}^{-1}M_n'x_{\mu}\subseteq x_{\mu}^{-1}M_1'x_{\mu}=M_{\mu}$. Thus we can apply Lemma \ref{thmlang} to obtain that each element of $u_nM_n'(\mathcal{O})\tau_{\mu}$ is $u_nM_n'(\mathcal{O})u_n^{-1}$-$\sigma$-conjugate to $u_n\tau_{\mu}$. Then by the last assertion in (b), $w=u_n$ is as desired.\\

We now prove (2). Each element of $Iw\tau_{\mu}I$ is obviously $I$-$\sigma$-conjugate to some element $g\in w\tau_{\mu}I$.  We have the Iwahori decomposition $I=N_{\mu}(\mathcal{O})I_{M_{\mu}}K_1$ where $I_{M_{\mu}}=I\cap M_{\mu}(\mathcal{O})$ and where $N_{\mu}$ is the unipotent radical of $P_{\mu}=M_{\mu}B$. We apply this to the last factor of $g\in w\tau_{\mu}I$. Now we use Section \ref{sec25} in the form $\epsilon^{\mu}N_{\mu}(\mathcal{O})\subseteq (N_{\mu}(\mathcal{O})\cap K_1)\epsilon^{\mu}$ and see that we can multiply $g$ by elements of $K_1$ on both sides to replace it by an element $g\in w\tau_{\mu}I_{M_{\mu}}$. Thus it is $I_{M_1}$-$\sigma$-conjugate to an element of $I_{M_1} w\tau_{\mu}=I_{M_1} wx_{\mu}\epsilon^{\mu}$. As $w\in{}^{\mu}W={}^{M_1}W$, conjugation by $w$ maps positive roots in $M_1$ to positive roots (not necessarily in $M_1$). Hence we have $w^{-1}I_{M_1}w\subset I\cap(w^{-1}{M_1}(\mathcal{O})w)$. Conjugation by $w_0$ maps all positive roots to negative roots, conjugation by $w_{0,{\mu}}$ maps positive roots in $M_{\mu}$ to negative roots in $M_{\mu}$ (and vice versa) and leaves positive roots in $N_{\mu}$ positive. Hence  
\begin{align*}
I_{M_{1}} wx_\mu\epsilon^{\mu} &=wx_{\mu}((wx_{\mu})^{-1}I_{M_1} wx_\mu)\epsilon^{\mu}\\
&\subseteq K_1wx_{\mu}(I_{M_{\mu}}\cap (wx_{\mu})^{-1}I_{M_1} wx_\mu))\overline{N}_{\mu}(\mathcal{O})\epsilon^{\mu}\\
&\subseteq K_1wx_{\mu}(I_{M_{\mu}}\cap (wx_{\mu})^{-1}I_{M_1} wx_\mu))\epsilon^{\mu}K_1.
\end{align*}  
Iterating this argument we see that the element $g$ is $I$-$\sigma$-conjugate to an element of $K_1wx_{\mu}I_{\infty}\epsilon^{\mu}K_1$ where $I_{\infty}=I\cap \bigcap_{i\geq 0}({\rm Ad}_{(wx_{\mu})^{-1}}\sigma^{-1})^{i}(M_{\mu}(\mathcal{O}))$ and where ${\rm Ad}_{(wx_{\mu})^{-1}}$ denotes conjugation with the given element. As $\bigcap_{i\geq 0}({\rm Ad}_{(wx_{\mu})^{-1}}\sigma^{-1})^{i}(M_{\mu})$ is an intersection of Levi subgroups it is equal to $\bigcap_{i= 0}^n({\rm Ad}_{(wx_{\mu})^{-1}}\sigma^{-1})^{i}(M_{\mu})$ for each sufficiently large $n$. Thus for the preceeding step it is in fact sufficient to $\sigma$-conjugate $g$ by finitely many elements. As $I_{\infty}$ commutes with $\epsilon^{\mu}$ and satisfies $(wx_{\mu})^{-1}\sigma^{-1}(I_{\infty})(wx_{\mu})=I_{\infty}$ we can apply Lemma \ref{thmlang} to obtain that each element of $K_1wx_{\mu}I_{\infty}\epsilon^{\mu}K_1$ is $I_{\infty}$-$\sigma$-conjugate to an element in $K_1w\tau_{\mu}K_1$.
\end{proof}

\section{Closure relations}\label{secclosure}
In this section we assume that $L=k((t))$, i.e.~we consider the function field case. Recall that $S_{w,\mu}$ is the locus in $LG$ where the truncation of level $1$ is equal to $(w,\mu)$.

\begin{lemma} \label{thmclosure}
\begin{enumerate}
\item Each $S_{w,\mu}$ is bounded and admissible.  
\item The closure $\overline{S_{w,\mu}}$ of $S_{w,\mu}$ is a union of finitely many strata $S_{w',\mu'}$.
\item $S_{w,\mu}$ is locally closed, smooth and irreducible.
\item $g_0\in G(L)$ is in $\overline{S_{w,\mu}}$ if and only if it is $K$-$\sigma$-conjugate to an element of $\overline{Iw\tau_{\mu}I}$.
\end{enumerate}
\end{lemma}

\begin{proof} 
The stratum is bounded because it is contained in $Kt^{\mu}K$, and admissible because it is invariant under $K_1$. For the second assertion note that $\overline{S_{w,\mu}}$ is invariant under $K$-$\sigma$-conjugation and under multiplication by $K_1$ on both sides. Thus it is a union of strata. The union is finite because $\overline{S_{w,\mu}}\subseteq \overline{Kt^{\mu}K}$, hence each of the strata $S_{w',\mu'}$ in the closure has to satisfy $\mu'\preceq\mu$. The first assertion of (3) follows from (1) and (2). The other two assertions of (3) follow as $S_{w,\mu}/K_1$ is the orbit under the $\sigma$-conjugation action of $K$ of the subscheme $K_1w\tau_{\mu}K_1$. In (4) the second condition implies the first by Theorem \ref{propdef} (2). Now let $g_0\in\overline{S_{w,\mu}}$. Thus there is a $g\in G(k[[z]]((t)))$ such that its reduction modulo $z$ is equal to $g_0$ and such that its image $g_{\eta}$ in $G(k((z))((t)))$ is in $S_{w,\mu}(k((z)))$. Hence there is an $h\in G(k((z))^{\rm alg}[[t]])$ with $h^{-1}g_{\eta}\sigma(h)\in K_{1,k((z))^{\rm alg}}w\tau_{\mu}K_{1,k((z))^{\rm alg}}$. Here $k((z))^{\rm alg}$ denotes an algebraic closure of $k((z))$. Replacing $h$ by a suitable element of $hK_{1,k((z))^{\rm alg}}$ we may assume that it is defined over a finite extension of $k((z))$. We may replace $k((z))$ by that totally ramified extension and thus assume that $h$ is defined over $k((z))$ itself. As $K/I\cong G(k)/B(k)$ is proper, there is a $k[[z]]$-valued point of $K/I$ such that the induced $k((z))$-valued point coincides with the image of $h$ in $K/I(k((z)))$. Let $\tilde h\in G(k[[z,t]])$ be a lift of that point. Such a lift exists because $k[[z]]$ is local, the map $G\rightarrow G/B$ has local sections, and we have the section $G(k((z)))\hookrightarrow K(k((z)))=G(k((z))[[t]])$ of the projection morphism $K\rightarrow G$. Denote by $\tilde{h}_{0}$ and $\tilde h_{\eta}$ the images of $\tilde{h}$ in $G(k[[t]])$ and $G(k((z))[[t]])$, respectively. As the generic points of $h$ and $\tilde h$ coincide up to an element of $I_{k((z))}$, we obtain that $\tilde h_{\eta}^{-1}g_{\eta}\sigma(\tilde h_{\eta})\in Iw\tau_{\mu}I$. Hence $\tilde h_{0}^{-1}g_0\tilde h_0\in\overline{Iw\tau_{\mu}I}$ which proves (4). 
\end{proof}
Before proving Theorem \ref{thmclos} we need some preparations. They are on the lines of \cite{He}, Section 3 where similar results are shown for finite Weyl groups and without the $\sigma$-action (but allowing disconnected groups).
\begin{remark}\label{remhe}
If $x,y,z\in\widetilde{W}$ with $x\in IyIzI$ then $x=y'z$ for some $y'\leq y$. Indeed, this follows by induction from $Is_iIzI\subseteq IzI\cup Is_izI$ for each (finite or affine) simple reflection $s_i$. 
\end{remark}

Let $\overline{N}$ be the unipotent radical of the Borel subgroup opposite of $B$. Let $\mathcal{N}^{-}$ be the inverse image of $\overline{N}$ in $G(k[t^{-1}])\subset G(k((t)))$, compare \cite{Faltings}, Section 2.
\begin{lemma}\label{lemhe33} Let $x,y\in\widetilde{W}$.
The subset $\{x'y\mid x'\leq x\}$ of $\widetilde{W}$ contains a unique minimal element $z$. We have $l(z)=l(y)-l(zy^{-1})$ and $\overline{IxIy\mathcal{N}^{-}}=\overline{Iz\mathcal{N}^{-}}$. In particular, $z\leq x'y'$ for every $x'\leq x$ and $y'\geq y$.
\end{lemma}
\begin{proof}
For any $x'\leq x$, $Ix'\subseteq \overline{IxI}$. Thus $\overline{Ix'y\mathcal{N}^{-}}\subseteq \overline{IxIy\mathcal{N}^{-}}$. We choose an increasing sequence $S_i$ of irreducible bounded subschemes of $\mathcal{N}^{-}$ with $\mathcal{N}^{-}=\bigcup_i S_i$. Recall from \cite{Faltings}, Section 3 that $I\backslash LG$ is the disjoint union of the $\mathcal{N}^{-}$-orbits of the elements $x\in\widetilde{W}$ and that $Ix_1\mathcal{N}^{-}\subseteq \overline{Ix_2\mathcal{N}^{-}}$ if and only if $x_2\leq x_1$. Note that $IxIyS_i$ is an irreducible bounded and admissible subscheme of $LG$. Let $y_i\in \widetilde{W}$ be the element whose orbit contains the generic point of $IxIyS_i$. Then $y_i\geq y_{i+1}$ for all $i$, hence $y_i=y_{i+1}$ for all sufficiently large $i$. Let $y_{\infty}$ be this element of $\widetilde{W}$. Then $\overline{IxIy\mathcal{N}^{-}}=\overline{Iy_{\infty}\mathcal{N}^{-}}$. As $Ix'y\mathcal{N}^{-}\subseteq \overline{Iy_{\infty}\mathcal{N}^{-}}$ we have that $x'y\geq y_{\infty}$ for all $x'\leq x$. 

It remains to show that $y_{\infty}=x_{\infty}y$ for some $x_{\infty}\leq x$ with $l(x_{\infty}y)=l(y)-l(x_{\infty})$. We use induction on the length of $x$. If $l(x)=0$, the statement is clear. Assume that $l(x)>0$. Let $s_i$ be a simple reflection with $s_ix<x$, and set $\xi=s_i x$. We have $$\overline{IxIy\mathcal{N}^{-}}=\overline{Is_iI\xi Iy\mathcal{N}^{-}}=\overline{Is_i\overline{I\xi Iy\mathcal{N}^{-}}}.$$
By induction there is a $\xi'\leq\xi$ such that $l(\xi' y)=l(y)-l(\xi')$ and $\overline{I\xi Iy\mathcal{N}^{-}}=\overline{I\xi' y\mathcal{N}^{-}}$. Thus 
$$\overline{Is_i\overline{I\xi Iy\mathcal{N}^{-}}}=\overline{Is_i\overline{I\xi' y\mathcal{N}^{-}}}=\overline{Is_iI\xi'y\mathcal{N}^{-}}=\begin{cases}
\overline{I\xi'y\mathcal{N}^{-}}& \text{if }s_i\xi'y>\xi'y\\
\overline{Is_i\xi'y\mathcal{N}^{-}}& \text{if }s_i\xi'y<\xi'y.
\end{cases}$$
We have $\xi'\leq s_i x<x$, thus $s_i\xi'\leq x$. If $s_i\xi'y>\xi'y$ we can choose $x_{\infty}=\xi'$. If $s_i\xi'y<\xi'y$ then $l(s_i\xi'y)=l(\xi' y)-1=l(y)-l(\xi')-1$. Thus $l(s_i\xi')=l(\xi')+1$ and $l(s_i\xi'y)=l(y)-l(s_i\xi')$, and we can choose $x_{\infty}=s_i\xi'$. Thus the assertion holds for $x$.
\end{proof}
\begin{lemma}\label{lemfacts}
\begin{enumerate}
\item If $a,b\in \widetilde{W}$ and $x\leq ab$ then there exist $a'\leq a$ and $b'\leq b$ with $a'b'=x$ and $l(a')+l(b')=l(x)$.
\item Let $M$ and $M'$ be standard Levi subgroups, $w\in \widetilde{W}^M\cap {}^{M'}\widetilde{W}$ and $v\in W_M$. Then $wv\in {}^{M'}\widetilde{W}$ if and only if $v\in{}^{K}W$ where $K=M\cap w^{-1}M'w$.
\end{enumerate}
\end{lemma}
\begin{proof}
For the first assertion, the general statement follows from the special case $x=ab$. This in its turn is a consequence of the exchange property of Coxeter groups using induction: If $a=\omega s_{i_1}\dotsm s_{i_r}$ with $\omega\in \Omega$ is a reduced expression for $a$ and $s_i$ a simple affine reflection then either $l(as_i)=l(a)+1$ or $as_i=\omega s_{i_1}\dotsm \hat s_{i_j}\dotsm s_{i_r}$ for some $j$. For a proof of the second statement see for example \cite{DDPW}, Theorem 4.18.
\end{proof}
\begin{lemma}\label{lemhe36}
Let $M$ and $M'$ be standard Levi subgroups, $w\in \widetilde{W}^M\cap {}^{M'}\widetilde{W}$ and $v\in W_M$. Let $K=M\cap w^{-1}M'w$. Then $wv=xwy$ for some $x\in W_{wKw^{-1}}$ and $y\in W_M\cap {}^K \widetilde{W}$.
\end{lemma}
\begin{proof}
By \cite{DDPW}, Theorem 4.18 we have $wv=xwy$ for some $x\in W_{M'}$ and $y\in  W_M\cap {}^K \widetilde{W}$. But then $x=wvy^{-1}w^{-1}\in wW_Mw^{-1}\cap W_{M'}=wW_{M\cap w^{-1}M'w}w^{-1}=W_{wKw^{-1}}$.
\end{proof}

\begin{lemma}\label{lemhe}
Let $M$ be a standard Levi subgroup of $G$ and let $x\in {}^M\widetilde{W}$. Let $y\in\widetilde{W}$. Then $y\geq wx\sigma(w)^{-1}$ for some $w\in W_{M}$ if and only if there are $u,v\in W_M$ with $v\leq u$ and $y\geq ux\sigma(v)^{-1}$.
\end{lemma}

\begin{proof}
Let $y\in\widetilde{W}$ and let $u,v\in W_M$ with $v\leq u$ and $y\geq ux\sigma(v)^{-1}$. We have to show that $y\geq wx\sigma(w)^{-1}$ for some $w\in W_{M}$. We use induction on the size of the Levi subgroup and thus may assume that the statement is true for all $M'\subsetneq M$. We use a second induction on the length $l(u)$. We write $x=ab$ with $a\in {}^M{\widetilde W}\cap {\widetilde{W}}^{\sigma(M)}$ and $b\in W_{\sigma(M)}$. Setting $M'=M\cap a\sigma(M)a^{-1}$ we decompose $u$ as $u_1u_2$ with $u_1\in W^{M'}$ and $u_2\in W_{M'}$. Together with $v\leq u$ this induces a decomposition $v=v_1v_2$ with $v_i\leq u_i$ and $l(v)=l(v_1)+l(v_2)$. Note that our choice of $a$ implies that $M'$ is again the Levi factor of a standard parabolic subgroup. We consider two cases:

\noindent \emph{Case 1: $u_1=v_1=1$.}
In this case $u$ and $v$ are in $W_{M'}$, and $x\in {}^{ M'}\widetilde W$. If $M'\neq M$, then the assertion follows from the induction hypothesis. If $M'=M=a\sigma(M)a^{-1}$, then since $ab\in {}^M\widetilde{W}$, we have that $b=1$. Thus $ux\sigma(v)^{-1}\geq x$ which implies the assertion. 

\noindent \emph{Case 2: $u_1\neq 1$.}
In this case $l(u_2)<l(u)$. By induction hypothesis, there is an $x'=u'x\sigma(u')^{-1}\leq u_2x\sigma(v_2)^{-1}$. Let $v_3\leq v_1$ be such that $x'\sigma(v_3)^{-1}$ is the unique element of minimal length in $\{x'\sigma(v')^{-1}\mid v'\leq v_1\}$ (see Lemma \ref{lemhe33}). Then the last assertion of Lemma \ref{lemhe33} implies that $x'\sigma(v_3)^{-1}\leq (u_2x\sigma(v_2)^{-1})\sigma(v_1)^{-1}= u_2x\sigma(v)^{-1}$. By Lemma \ref{lemhe36} we can write $x\sigma(v)^{-1}=a(b\sigma(v)^{-1})\in ({}^M{\widetilde W}\cap {\widetilde{W}}^{\sigma(M)})W_{\sigma(M)}$ as $\alpha a\delta$ with $\alpha\in W_{M'}$ and $\delta\in W_{\sigma(M)}\cap {}^{a^{-1}M'a}W$. By Lemma \ref{lemfacts}(2) $\beta=a\delta\in {}^M\widetilde{W}$. Thus $l(u_1u_2x\sigma(v)^{-1})=l(u_1u_2\alpha\beta)=l(u_1u_2\alpha)+l(\beta)=l(u_1)+l(u_2\alpha)+l(\beta)=l(u_1)+l(u_2\alpha\beta)=l(u_1)+l(u_2x\sigma(v)^{-1})$. As $x'\sigma(v_3)^{-1}\leq u_2x\sigma(v)^{-1}$ and $v_3\leq v_1\leq u_1$, this implies that  $(v_3u')x\sigma(v_3u')^{-1}=v_3x'\sigma(v_3)^{-1}\leq ux\sigma(v)^{-1}\leq y$.
\end{proof}

\begin{proof}[Proof of Theorem \ref{thmclos}]
We have to show that $(w',\mu')$ is the truncation of level 1 of an element of $IyI$ for some $y\leq w\tau_{\mu}$ if and only if it is of the form in the theorem. The \emph{if} part is obvious. For the other direction we use an approach which is similar to the proof of Theorem \ref{propdef} to compute the truncations of level 1 occurring in the cosets $IyI$ for $y$ as above. We decompose $y$ as $w_1\tau_{\mu'}w_1'$ with $w_1,w_1'\in W$, $\mu'$ dominant and such that the lengths of the three elements add up to that of $y$. Each truncation of an element of $IyI$ already occurs in $Iy=Iw_1\tau_{\mu'}w_1'$, and thus also in $\sigma^{-1}(w_1')Iw_1\tau_{\mu'}$. By Remark \ref{remhe}, each such element is contained in $I\sigma^{-1}(\tilde{w}_1')w_1\tau_{\mu'}I$ for some $\sigma^{-1}(\tilde{w}_1')\leq \sigma^{-1}(w_1')$. This is equivalent to $\tilde{w}_1'\leq w_1'$ as $I$ is invariant under $\sigma$. Using Lemma \ref{lemfacts}(1) for $\sigma^{-1}(\tilde{w}_1')w_1$ and replacing $y$ by a smaller element we see that we may assume that $\tilde w_1'=w_1'$ and that $l(\sigma^{-1}(w_1')w_1\tau_{\mu'})=l(w_1\tau_{\mu'}w_1')=l(w_1)+l(\tau_{\mu'})+l(w_1')$. We have to consider the truncation types occurring in $\sigma^{-1}(w_1')w_1\tau_{\mu'}I$. It is enough to show that for each such type $(w',\mu')$ there is a $u\in W$ with $uw'\tau_{\mu'}\sigma(u)^{-1}\leq \sigma^{-1}(w_1')w_1\tau_{\mu'}$. Indeed, by Lemma \ref{lemfacts} (1) this implies that there is a $v_1\leq \sigma^{-1}(w_1')$ such that $v_1^{-1}uw'\tau_{\mu'}\sigma(u)^{-1}\sigma(v_1)\leq w_1\tau_{\mu'}w_1'$. By Lemma \ref{lemhe}, it is furthermore enough to show the following claim.

{\it Claim.} Let $(w',\mu')$ be the truncation of level 1 of an element $g\in Ix\tau_{\mu'}I$ for some $x\in W$. Then there are $v\leq u\in \sigma^{-1}(W_{M_{\mu'}})$ with $uw'\tau_{\mu'}\sigma(v)^{-1}=x\tau_{\mu'}$. 

By $\sigma$-conjugating with the first factor of $g$ we may assume that it is contained in $x\tau_{\mu'}I$. Changing $g$ within its $K_1$-double coset we may assume that the factor in $I$ is in fact contained in $I\cap B(\mathcal{O})\cap \overline{P}_{\mu'}(\mathcal{O})\subseteq I\cap M_{\mu'}(\mathcal{O})$. A second $\sigma$-conjugation then implies that we may assume that $g\in (I\cap M_1(\mathcal{O}))x\tau_{\mu'}$ where $M_1=\sigma^{-1}(M_{\mu'})$ is as in the proof of Theorem \ref{propdef}. Note that for the groups defined in that proof we have $M_i'\subseteq M_1'$ and hence $u_iM_{i+1}(\mathcal{O})u_i^{-1}\subseteq M_1(\mathcal{O})$. In particular, the construction in this proof implies for the element $g\in (I\cap M_1(\mathcal{O}))x\tau_{\mu'}$ that there is an $f\in M_1(\mathcal{O})$ with $f^{-1}g\sigma(f)\in K_1 w'\tau_{\mu'}K_1$. We decompose $f$ as $f=i_1ui_2\in IW_{\sigma^{-1}M_{\mu'}}I$. Then $i_1ui_2 w'\tau_{\mu'}\sigma(i_1ui_2)^{-1}\in Ix\tau_{\mu'}I$. Recall that $\tau_{\mu'}$ is the shortest representative of its $W$-double coset and $w'\in {}^{\sigma^{-1}(M_{\mu'})}W$. Thus $w'\tau_{\mu'}\in {}^{\sigma^{-1}(M_{\mu'})}\widetilde W$. Hence  
$Ix\tau_{\mu'}I\subseteq IuI w'\tau_{\mu'}I\sigma(u)^{-1}I=Iu w'\tau_{\mu'}I\sigma(u)^{-1}I$. Thus there is a $v\leq u$ with $uw'\tau_{\mu'}\sigma(v)^{-1}=x\tau_{\mu'}$. 

\end{proof}

The following corollary to the theorem which considers the special case $\mu=\mu'$ is analogous to results by He \cite{He} and Wedhorn \cite{Wedhorn}.
\begin{kor}\label{corclosure}
$S_{w',\mu}\subseteq  \overline{S_{w,\mu}}$ if and only if there is a $\tilde w\in \sigma^{-1}(W_{M_\mu})$ with $\tilde{w}^{-1}w'x_{\mu}\sigma(\tilde{w})x_{\mu}^{-1}\leq w$. 
\end{kor}

\begin{proof}
Recall that $\tau_{\mu}$ is the unique shortest element of the extended affine Weyl group lying in $Wt^{\mu} W$. Especially, $y\leq w\tau_{\mu}$ with $y\in Wt^{\mu} W$ if and only if $y=w_y\tau_{\mu}$ for some $w_y\leq w$ in $W$. From the theorem we obtain that $S_{w',\mu}\subseteq  \overline{S_{w,\mu}}$ if and only if there is a $\tilde w\in W$ such that $\tilde{w}^{-1}w'\tau_{\mu}\sigma(\tilde{w})=w_y\tau_{\mu}$ for some $w_y$ as above. Thus $\tau_{\mu}\sigma(\tilde{w})=v\tau_{\mu}$ for some $v\in W$. As $\tau_{\mu}=x_{\mu}t^{\mu}$ we obtain $v=x_{\mu}\sigma(\tilde{w})x_{\mu}^{-1}$ and $\sigma(\tilde{w})\in W_{M_\mu}$.
\end{proof}

\section{Non-emptiness of intersections of truncation strata with $\sigma$-conjugacy classes}\label{secproof}

For the discussion of short elements we allow both possible cases for $F$. 

\begin{definition}\label{defshort}
Let $[b]\in B(G)$ and let $\nu\in X_*(T)_{\mathbb{Q}}^{\Gamma}$ be its dominant Newton point. Let $M_{\nu}$ be the centralizer of $\nu$ in $G$. Then $x\in\widetilde{W}$ is called \bsh if $x\in \Omega_{M_{\nu}}\subseteq \widetilde{W}_{M_{\nu}}$, the ${M_{\nu}}$-dominant Newton point of $x$ is equal to $\nu$, and $\kappa_G(x)=\kappa_G(b)$.

An element $x\in\widetilde{W}$ is called \sh if it is \bsh for some $b\in B(G)$.
\end{definition}
\begin{remark}
From the classification of $B(G)$ we obtain that all \bsh elements are contained in $[b]$.
\end{remark}
\begin{lemma}\label{lemshort}
Each $[b]\in B(G)$ contains a \bsh element. If $G$ is split, this element is unique.
\end{lemma}
\begin{proof}
Let $\nu\in X_*(T)_{\mathbb{Q}}^{\Gamma}$ be the dominant Newton point of $b$ and let $M$ be the centralizer of $\nu$ in $G$. Then there is an element $b_0$ of $M(L)\cap [b]$ whose $M$-dominant Newton point is equal to $\nu$ (\cite{Kottwitz1}, Proposition 6.2). Let $\mu_0\in X_*(\T)$ be $M$-dominant with $b_0\in M(\mathcal{O})\epsilon^{\mu_0} M(\mathcal{O})$. Let $\omega$ be the image of $\mu_0$ in $\pi_1(M)$. Note that its image under the projection to $\pi_1(G)_{\Gamma}$ is equal to $\kappa_G(b)$. Let $x\in \Omega_M$ be the unique element whose image under the isomorphism to $\pi_1(M)$ is equal to $\omega$. Then $x$ is basic in $M$ with $\kappa_M(x)=\kappa_M(b_0)$, hence with $M$-dominant Newton point $\nu$. In particular, $x$ is $[b]$-short. 

For split $G$, we have $\pi_1(G)=\pi_1(G)_{\Gamma}$. The kernel of the projection $\pi_1(M)\rightarrow \pi_1(G)$ is torsion free. Hence $\omega\in\pi_1(M)$ is the unique element whose image in $\pi_1(G)$ is equal to $\kappa(b)$ and whose image in $\pi_1(M)\otimes\mathbb{Q}$ is equal to the image of $\nu$ under the projection to $\pi_1(M)\otimes\mathbb{Q}$. Each element $b'$ of $[b]\cap M(L)$ whose $M$-dominant Newton point is equal to $\nu$ has to satisfy $\kappa_M(b')=\omega$. Thus, there is a unique such element which lies in $\Omega_M$.
\end{proof}

For the rest of this section let $L=k((t))$, i.e.~we consider the function field case.
\begin{remark}\label{remit}
By Theorem \ref{propdef} (2), $S_{w,\mu}$ has nonempty intersection with some $\sigma$-conjugacy class $[b]$ if and only if $[b]\cap Iw\tau_{\mu}I\neq\emptyset$. By the Grothendieck specialization theorem \cite{RapoportRichartz}, Theorem 3.6, the generic $\sigma$-conjugacy class in $S_{w,\mu}$ resp. the generic class in $Iw\tau_{\mu}I$ are the largest classes (with respect to $\preceq$) whose intersections with $S_{w,\mu}$ resp. $Iw\tau_{\mu}I$ are non-empty. Hence also these generic classes coincide.
\end{remark}

\begin{prop}\label{propmain}
Let $b\in G(L)$ and let $M$ be the centralizer of its dominant Newton point. Let $x\in\widetilde{W}$ with $b\in IxI$. Then there is a \bsh element $x_b\in\widetilde{W}$ and a $w\in {}^MW$ with $w^{-1}x_b \sigma(w)\leq x$ in the Bruhat order.
\end{prop}

\begin{proof}
Let $P=BM$ be the standard parabolic subgroup of $G$ with Levi component $M$, and let $N$ be its unipotent radical. We fix a \bsh element $y\in \widetilde{W}_M$. Let $g\in G(L)$ with $g^{-1}y\sigma(g)=b \in IxI$. Using the Iwasawa decomposition and the Bruhat decomposition we write $g=nmi_1wi_2$ with $n\in N(L)$, $m\in M(L)$, $i_1,i_2\in I$, and $w\in W$. By the Iwahori decomposition, $i_1\in P(\mathcal{O})K_1$. As $w^{-1}K_1w\subseteq I$, we may assume that $i_1=\id$. Furthermore we can replace $g$ by $gi_2^{-1}$ without changing the property $g^{-1}y\sigma(g)\in IxI$. Thus we may assume that $g=nmw$ with $g^{-1}y\sigma(g)\in Ix I$. Finally we may assume that $w$ is of minimal length in its coset $W_{M}w$. 

The next step is to show that we may assume that $n=1$, i.e.~that $w^{-1}m^{-1}y\sigma(mw)\in \overline{IxI}$. We have
\begin{equation}\label{glg}
g^{-1}y\sigma(g)=w^{-1}m^{-1}\left[n^{-1}y\sigma(n)y^{-1}\right]y\sigma(mw).
\end{equation}
We abbreviate the expression in the bracket, which is in $N(L)$, by $\tilde{n}$. We want to construct a family of elements of $\overline{IxI}$ over $\mathbb{A}^1_k$ such that its fiber over $1$ is $g^{-1}y\sigma(g)$, and that the fiber over $0$ is $w^{-1}m^{-1}y\sigma(mw).$ Let $LN$ be the loop group associated with $N$ over $k$, i.~e. the group ind-scheme representing the functor on $k$-algebras $R\mapsto N(R((t)))$. Let $\chi\in X_*(\T)$ be central in $M$ and such that $\langle\alpha,\chi\rangle>0$ for every simple root $\alpha$ of $T$ in $N$. Let
\begin{eqnarray*}
\phi: \mathbb{A}_{k}^1\setminus\{0\} &\rightarrow &LN\\
a&\mapsto& \chi(a)\tilde{n}\chi(a)^{-1}.
\end{eqnarray*}
Let $\alpha$ be a root of $\T$ in $N$ and let $U_{\alpha}$ denote the corresponding root subgroup. Conjugation by $\chi(a)$ maps $U_{\alpha}(y)$ to $U_{\alpha}(a^jy)$ where $j=\langle\alpha,\chi\rangle>0$. Especially, $\phi$ has an extension to a morphism $\phi:\mathbb{A}_k^1\rightarrow LN$ that maps $0$ to $\id$. As $\chi(a)$ is central in $M$, 
$$
w^{-1}m^{-1}\phi(a)y\sigma(mw)=(w^{-1}\chi(a)w)w^{-1}m^{-1}\tilde{n}y\sigma(mw)(\sigma(w)^{-1}\chi(a)^{-1}\sigma(w))
$$
for every $a\neq 0$. Using \eqref{glg}, we obtain that this is in $Ix I$. Hence $$w^{-1}m^{-1}\phi(0)y\sigma(mw)=w^{-1}m^{-1}y\sigma(mw)\in \overline{IxI}.$$ 

It remains to show that $w^{-1}m^{-1}y\sigma(mw)\in \overline{IxI}$ implies that $w^{-1}x_b\sigma(w)\in \overline{IxI}$ for some \bsh element $x_b$. Let $I_{M}=I\cap M(L)$. The minimality property of $w$ implies that for any positive root $\alpha$ of $\T$ in $M$ the root $\beta$ with $w^{-1}U_{\alpha}w=U_{\beta}$ is also positive (although not necessarily in $M$). As $I$ and $M$ are defined over $\mathcal{O}_F$, the same holds for $\sigma(w)$. Thus $$w^{-1}I_{M}m^{-1}y\sigma(m)I_{M}\sigma(w)\subseteq Iw^{-1}m^{-1}y\sigma(mw)I\subseteq \overline{IxI}.$$ Using the Cartan decomposition for $M$ we have $m^{-1}y\sigma(m)\in M(\mathcal{O}_L)\epsilon^{\mu'}M(\mathcal{O}_L)$ for some $M$-dominant $\mu'\in X_*(T)$. Let $x_b\in \Omega_M\subseteq\widetilde{W}_M$ be the unique element whose image under the projection $pr_M:\widetilde{W}_M\rightarrow\pi_1(M)$ agrees with the image of $\mu'$. In particular, this implies that $\kappa_M(x_b)=\kappa_M(m^{-1}y\sigma(m))=\kappa_M(y)\in \pi_1(M)_{\Gamma}$. As $x_b$ is basic in $M$, the $M$-dominant Newton polygons of $x_b$ and $y$ agree. Thus $x_b$ is a $[b]$-short element. As $x_b\in\Omega_M$ we have that $I_{M}x_bI_{M}$ is the unique closed $I_{M}$-double coset of the form $I_{M}hI_{M}$ with $h\in \widetilde{W}_M$ and $pr_M(h)=pr_M(x_b)$. It is contained in the closure of any other such double coset. Applying this to $I_{M}m^{-1}y\sigma(m)I_{M}$ we obtain $$w^{-1}I_{M}x_bI_{M}\sigma(w)\subseteq \overline{IxI}.$$ This implies $w^{-1}x_b\sigma(w)\in\overline{IxI}$.
\end{proof}

If $G$ is split, the first (and largest) part of the proof of Proposition \ref{propmain} can also be deduced from a non-emptiness criterion by G\"ortz, Haines, Kottwitz, and Reuman, \cite{GHKR2}, Corollary 12.1.2 using the relation between \sh elements and fundamental alcoves in Lemma \ref{lemfundshort}. 

\begin{kor}\label{kormain}
Let $[b_x]$ be the generic $\sigma$-conjugacy class in $IxI$ for some $x\in \widetilde{W}$. Then $[b_x]$ is the unique largest (with respect to the order described in the introduction) among the classes $[y]$ where $y\in\widetilde{W}$ with $y\leq x$ in the Bruhat order. It is also equal to the largest among the $[y]$ where $y\leq x$ is in addition of the form $y=w^{-1}z\sigma(w)$ where $z$ is $[y]$-\sh and where $w\in {}^{M_{y}}W$ for the centralizer $M_y$ of the dominant Newton point of $y$.
\end{kor}
\begin{proof}
The generic $\sigma$-conjugacy classes of $IxI$ and $\overline{IxI}$ coincide. By the Grothendieck specialization theorem \cite{RapoportRichartz}, Theorem 3.6, the generic class of $\overline{IxI}$ is the unique largest (with respect to the order described in the introduction) among the classes $[g]$ with $g\in \overline{IxI}$. Hence the assertion follows from Proposition \ref{propmain}.
\end{proof}

\begin{proof}[Proof of Theorem \ref{thmmain} and Corollary \ref{corgennp}]
Theorem \ref{thmmain} and Corollary \ref{corgennp} follow from Proposition \ref{propmain} and Corollary \ref{kormain} by Remark \ref{remit}.
\end{proof}

\begin{prop}\label{propcentralstream}
Let $(w,\mu)$  be the truncation type of a \bsh element for some $\sigma$-conjugacy class $[b]$. If a $\sigma$-conjugacy class $[b']$ is contained in $\overline{[b]}$ then there exists a $[b']$-short element $x'$ such that $S_{w',\mu'}\subseteq \overline{S_{w,\mu}}$ where $(w',\mu')=\tr(x')$. If $[b]=[b_{w\tau_{\mu}}]$ then the converse also holds. This is in particular always the case if $G$ is split.
\end{prop}
The closure of $[b]$ is a union of $\sigma$-conjugacy classes. By \cite{RapoportRichartz}, Theorem 3.6 a necessary condition for $[b']\subseteq\overline{[b]}$ is that $[b']\preceq[b]$, i.e.~$\kappa_G(b)=\kappa_G(b')$ and $\nu_{b'}\preceq\nu_{b}$. In \cite{grothconj} it is shown that for split $G$ this condition is also sufficient.

\begin{proof}
Assume that $[b']\subseteq\overline{[b]}$. Let $g\in G(k[[z]]((t)))$ such that $g_{k((z))}\in [b]$ and $g_k\in [b']$. Let $h\in G(k((z))^{\rm alg}((t)))$ with $h^{-1}g_{k((z))}\sigma(h)\in Iw\tau_{\mu}I$. Here $k((z))^{\rm alg}$ denotes an algebraic closure of $k((z))$. The closed Schubert cell in $LG/I$ containing $h$ is a scheme of finite type. Thus replacing $h$ by some representative of $hI_{k((z))^{\rm alg}}$ we may assume that $h$ is defined over a finite extension of $k((z))$. Replacing $k[[z]]$ by its integral closure in that extension we may assume $h\in G(k((z))((t)))$. Also, as the closed Schubert cell is a proper subscheme of $LG/I$, the class $hI$ contains an element of $LG/I(k[[z]])$. As $k[[z]]$ is local for the \'etale topology, \cite{foliat}, Lemma 2.3 shows that we obtain an element of $LG(k[[z]])=G(k[[z]]((t)))$ in the inverse image. We denote this element again by $h$. Then $\tilde g=h^{-1}g\sigma(h)\in G(k[[z]]((t)))$ with $\tilde g_{k((z))}\in Iw\tau_{\mu}I\subseteq S_{w,\mu}$ and $\tilde g_k\in [b']$. Hence $\overline{S_{w,\mu}}$ contains an element of $[b']$ and thus by Theorem \ref{thmmain} also some stratum $S_{w',\mu'}$.

Let now $[b]=[b_{w\tau_{\mu}}]$ and assume that there exists a $[b']$-short element $x'$ such that $S_{w',\mu'}\subseteq \overline{S_{w,\mu}}$. Then $[b']\cap \overline{[b]}\neq\emptyset$, hence $[b']\subseteq \overline{[b]}$. 

It remains to show that for split $G$ we always have $[b]=[b_{w\tau_{\mu}}]$. Let $[b']=[b_{w\tau_{\mu}}]$. Then $[b]\cap \overline{[b']}\neq\emptyset$, hence $[b]\subseteq \overline{[b']}$. Let $(w',\mu')$ be the truncation type of the unique $[b']$-short element (compare Lemma \ref{lemshort}). Then by the first assertion of this proposition, we have $S_{w,\mu}\subseteq \overline{S_{w',\mu'}}$. On the other hand, $[b']\cap Iw\tau_{\mu}I\neq\emptyset$, thus by Proposition \ref{propmain}  $S_{w',\mu'}\subseteq \overline{S_{w,\mu}}$. Thus $w=w'$, $\mu=\mu'$, and $[b]=[b']$.
\end{proof}

\begin{remark}
Essentially the same proof also shows the following statement. Let $b,b'\in G(L)$ and let $x\in \widetilde{W}$ such that $IxI\subseteq [b]$ (for example a $P$-fundamental alcove contained in $[b]$ as in Theorem \ref{thmexfundalc}). Then $[b']\subseteq\overline{[b]}$ if and only if $[b']\cap \overline{IxI}\neq\emptyset$. 
\end{remark}

\section{Comparison between the arithmetic case and the function field case}\label{secfundalc}
In this section we consider both cases $L=k((t))$ and $L=\Quot (W(k))$, and compare between them.
\begin{definition}\label{defpfund}
\begin{enumerate}
\item For $x\in G(L)$ let $\phi_x:G(L)\rightarrow G(L)$ with $g\mapsto \sigma(xgx^{-1})$.
\item Let $P$ be a semistandard parabolic subgroup of $G$, i.e.~a parabolic subgroup containing $\T$ but not necessarily $B$. Let $N$ be its unipotent radical and $M$ the Levi factor containing $\T$. Let $\overline{N}$ be the unipotent radical of the opposite parabolic. Then an element $x\in \widetilde{W}$ is called $P$-fundamental if $\phi_x(I_M)=I_M$, $\phi_x(I_N)\subseteq I_N$, and $\phi_x(I_{\overline{N}})\supseteq I_{\overline{N}}$.
\end{enumerate}\end{definition}
This definition is a generalization of G\"ortz, Haines, Kottwitz, and Reuman's notion of fundamental $P$-alcoves for split groups from \cite{GHKR2}, 13. Also, Lemma \ref{remoortmain}, Theorem \ref{thmexfundalc} and Proposition \ref{lemsuperset}(1) are generalizations to unramified groups of corresponding results of \cite{GHKR2}. However, for our main theorem in this context (Theorem \ref{thmexfundalc}) one needs a different proof than the one used for split groups. 

\begin{remark}\label{remnuM}
Let $x\in\widetilde{W}$. Let $r>0$ be such that $G$ is split over an unramified extension of $\mathcal{O}_F$ of degree $r$. Hence $\sigma^r$ acts trivially on $\widetilde W$. Let $x':=\sigma(x)\sigma^2(x)\dotsm\sigma^r(x)$. Let $P=MN$ be a semistandard parabolic subgroup with $\phi_x(P)=P$. Then $x'P(x')^{-1}=\phi_x^r(P)=P$, hence $x'\in \widetilde W_M$. We denote the $M$-dominant Newton point of the $\sigma^r$-conjugacy class of $x'$ by $\nu_{r,M}$.
\end{remark}
\begin{lemma}\label{lemnew}
Let $x\in [b]\cap \widetilde{W}$ be $P$-fundamental for some $P=MN$. Let $r$ and $x'$ be as in Remark \ref{remnuM}. Then $x$ is $Q$-fundamental for a semistandard parabolic $Q=M_QN_Q$ if and only if
\begin{enumerate}
\item[(i)]$\phi_x(Q)=Q$,
\item[(ii)] $\nu_{r,M}$ is central in $M_Q$ and
\item[(iii)] $\langle \nu_{r,M},\alpha\rangle\geq 0$ for each root $\alpha$ of $T$ in $N_Q$.
\end{enumerate}
\end{lemma}
\begin{proof}
Let $P=MN$ where $M$ is the Levi factor of $P$ containing $\T$. As $x$ is $P$-fundamental, $\phi_x$ stabilizes $M$ and $N$, hence (i) holds for $P$. Besides, $I_M=\phi_x^r(I_M)=x'I_M(x')^{-1}$. Hence $x'\in\Omega_M$, and the Newton point $\nu_{r,M}$ is central in $M$. Similarly, $\phi_x(I_N)\subseteq I_N$ implies condition (iii) for $P$. Indeed, let $r'>0$ be such that $(x')^{r'}\in X_*(T)\subseteq \widetilde W$. For example, $r'$ can be chosen to be the order of the factor in $W$ of $x'\in\widetilde W\cong W\ltimes X_*(T)$. Then $(x')^{r'}=r'\nu_{r,M}$, and (iii) follows using Section \ref{sec25}. To prove the converse we first consider a special case. Let $M'$ be the centralizer of $\nu_{r,M}$. Then $M\subseteq M'$. Let $P'$ be the parabolic subgroup generated by $M'$ and $N$. Let $N'$ be its unipotent radical. Then $\phi_x(P')=P'$, and $P'\supseteq P$ and $N'\subseteq N$. Thus in order to show that $x$ is $P'$-fundamental it is enough to verify $\phi_x(I_{M'})=I_{M'}$. We consider the decomposition $I_{M'}=I_MI_{N\cap M'}I_{\overline N\cap M'}$. Each of the subgroups $M,N\cap M',\overline N\cap M'$ is stable under $\phi_x$, so we can consider each factor of $I_{M'}$ separately. For $I_M$ the assertion is just the assumption. For $I_{N\cap M'}$ we have $\phi_x^i(I_{N\cap M'})\subseteq I_{N\cap M'}$ for every $i> 0$. For $i=rr'$ (with $r'$ as above) we have equality in the above containment. Indeed, $\phi_x^{rr'}(g)=\sigma^{rr'}((x')^{r'}g(x')^{-r'})$ and $(x')^{r'}\in X_*(T)$ with $\langle \alpha,(x')^{r'}\rangle=\langle \alpha,r'\nu_{r,M}\rangle=0$ for every root of $\T$ in $M'$. Considering the whole chain of containments we obtain equality for every $i$. A similar argument applies to $\overline N\cap M'$. It remains to show that if $Q \subseteq P'$ satisfies (i)--(iii), then $x$ is $Q$-fundamental, but this is obvious.
\end{proof}

\begin{lemma}\label{remoortmain}
If $x$ is $P$-fundamental then every element of $IxI$ is $I$-$\sigma$-conjugate to $x$.
\end{lemma}

\begin{proof}
Each element of $IxI$ is $I$-$\sigma$-conjugate to an element of $xI$. By Lemma \ref{lemnew} we may assume that $P$ is maximal with the property that $x$ is $P$-fundamental, i.e.~that $\langle \nu_{r,M},\alpha\rangle>0$ for each root $\alpha$ of $T$ in $N$. In particular $\phi_x$ is topologically nilpotent on $I_N$, see Section \ref{sec25}.

Let $g\in I$. We apply the Iwahori decomposition to $g$ to obtain $g=g_Ng_Mg_{\overline{N}}\in I_NI_MI_{\overline{N}}$. Note that $xg_Nx^{-1}\in \sigma^{-1}(I)=I$. Thus $xg=(xg_Nx^{-1})xg_Mg_{\overline N}$ is $I$-$\sigma$-conjugate to $xg_Mg_{\overline N}\phi_x(g)\in x(I\cap \phi_x(I))$. By the Iwahori decomposition $I\cap \phi_x(I)=\phi_x(I_{N})I_MI_{\overline{N}}$. Using this to decompose $g_Mg_{\overline N}\phi_x(g)$ and iterating we obtain that $xg$ is $I$-$\sigma$-conjugate to an element of $x(I\cap \phi_x^n(I))$ for every $n$. Note that in the $n$th iteration we only $\sigma$-conjugate by an element of $x\phi_x^n(I_N)x^{-1}$. The morphism $\phi_x$ is topologically nilpotent on $I_N$. Hence the product of these elements exists and in the limit we obtain that $xg$ is $I$-$\sigma$-conjugate to an element of $x\left(\bigcap_{n\geq 0} \phi^n_x(I)\right)=xI_MI_{\overline{N}}$. We write this element as $xg'_Mg'_{\overline{N}}$. It is $I$-$\sigma$-conjugate to $x(x^{-1}\sigma^{-1}(g'_{\overline N})x)g'_M=xg'_M((g'_M)^{-1}\phi_x^{-1}(g'_{\overline N})g'_M)\in x(I\cap \phi_x^{-1}(I)\cap M\overline N)$. A similar iteration as above shows that $xg'_Mg'_{\overline{N}}$ is $I$-$\sigma$-conjugate to an element of $xI_M$. By our assumption $h\mapsto x^{-1}\sigma^{-1}(h^{-1})xh=\phi_x^{-1}(h^{-1})h$ defines a morphism $I_M\rightarrow I_M$. By Lemma \ref{thmlang} it is surjective, hence for every $g\in I_M$ there is a $\sigma^{-1}(h)\in\sigma^{-1}(I_M)\subseteq I$ which $\sigma$-conjugates $x$ to $xg$.
\end{proof}

\begin{thm}\label{thmexfundalc}
For every $[b]\in B(G)$ there exists an $x\in \widetilde W$ such that $x\in [b]$ is $P$-fundamental for some semi-standard $P$.
\end{thm}
Note that it is not easy to give an explicit description of a $P$-fundamental alcove contained in a given $[b]$. In general, neither \bsh elements nor the representatives $w\tau_{\mu}$ of their truncation types $(w,\mu)$ are $P$-fundamental for any $P$.

For the proof of the theorem we need the following three lemmas.

\begin{lemma}\label{lemcarmin1}
Let $P$ be a semistandard parabolic subgroup of $G$. Let $\tilde I$ be an Iwahori subgroup of $LG$ containing $T(\mathcal{O})$. Let $\tilde I_M=\tilde I\cap M$ and similarly for $\tilde I_N$ and $\tilde I_{\overline{N}}$. Let 
$x\in\widetilde W$ with $\phi_x(\tilde I_M)=\tilde I_M$. Then $\phi_x(\tilde I_N)\subseteq \tilde I_N$ if and only if $\phi_x(\tilde I_{\overline N})\supseteq \tilde I_{\overline N}$.
\end{lemma}
Note that in this lemma we do not assume $\tilde I$ or $P$ to be fixed by $\sigma$.

\begin{proof}
As $\tilde I$ contains $T(\mathcal{O})$, the group $\tilde I_{\overline{N}}$ is a product of its intersections with the root subgroups for roots of $T$ in $\overline{N}$. Let $U_{\alpha}$ be such a root subgroup. We write $x=\epsilon^{\mu_x}w_x$ with $\mu_x\in X_*(T)$ and $w_x\in W$. Let $\psi(\alpha)$ be the root of $T$ in $\overline{N}$ with $\sigma(w_xU_{\psi(\alpha)}w_x^{-1})=U_{\alpha}$. The assertion on $\tilde I_{\overline N}$ is equivalent to $U_{\alpha}\cap \tilde I_{\overline{N}}\subseteq \sigma (x\tilde I_{\overline N}x^{-1})$ for all $\alpha$. This is equivalent to $U_{\alpha}\cap \tilde I_{\overline{N}}\subseteq \sigma (x(U_{\psi(\alpha)}\cap \tilde I_{\overline{N}})x^{-1})$. Note that $U_{\alpha}\cong \mathbb{G}_a$, hence we can identify $U_{\alpha}\cap \tilde I$ with $\epsilon^{\phi_{\alpha}}k[[\epsilon]]$ for some $\phi_{\alpha}\in\mathbb{Z}$. As $x=\epsilon^{\mu_x}w_x$ the above inclusion holds if and only if 
\begin{equation}\label{eqalpha}
\langle \mu_x,\sigma^{-1}(\alpha)\rangle+\phi_{\psi(\alpha)}\leq \phi_{\alpha}
\end{equation}
for all $\alpha$. As $\tilde I$ is an Iwahori subgroup we have $\phi_{\alpha}+\phi_{-\alpha}=1$ for all $\alpha$. Hence \eqref{eqalpha} is equivalent to
$$\langle \mu_x,\sigma^{-1}(-\alpha)\rangle+\phi_{-\psi(\alpha)}\geq \phi_{-\alpha}.$$ Note that $-\psi(\alpha)=\psi(-\alpha)$. Hence this last inequality is equivalent to the inclusion $\sigma(x\tilde I_Nx^{-1})\subseteq \tilde I_N$.
\end{proof}

\begin{lemma}\label{lemhelp2}
Let $[b]\in B(G)$ and let $M_{\nu}$ be the centralizer of its dominant Newton point. Let $P_{\nu}$ be the associated parabolic subgroup and let $N_{\nu}$ be its unipotent radical. Then $[b]$ contains a $P$-fundamental alcove $x$ if and only if there is an Iwahori subgroup $\tilde I$ with $\tilde I\cap M_{\nu}=I\cap M_{\nu}$ and an element $b_0\in \Omega_{M_{\nu}}\cap [b]$ with $\sigma(b_0(\tilde I\cap N_{\nu})b_0^{-1})\subseteq \tilde I\cap N_{\nu}$.
\end{lemma}
\begin{proof}
Assume first that there is an Iwahori subgroup $\tilde I$ and a $b_0$ as above. As $b_0\in \Omega_{M_{\nu}}$ we have $\phi_{b_0}(I_{M_{\nu}})=\sigma(I_{M_{\nu}})=I_{M_{\nu}}$. By Lemma \ref{lemcarmin1}, $\phi_{b_0}(\tilde I\cap \overline{N}_{\nu})\supseteq \tilde I\cap \overline N_{\nu}$. Note that $\tilde I\cap M_{\nu}=I\cap M_{\nu}$ implies that $T(\mathcal{O})\subset \tilde I$. Let $y\in \widetilde W$ with $y^{-1}\tilde I y=I$. Let $M=y^{-1}M_{\nu}y$ and $P=y^{-1}P_{\nu}y$ with unipotent radical $N$ and opposite $\overline N$. Then $y^{-1}\tilde I_{M_{\nu}}y=I_M$. Let $x=\sigma^{-1}(y)^{-1}b_0y\in [b]$. We have
\begin{eqnarray*}
\sigma(x I_{M}x^{-1})&=&\sigma(xy^{-1}\tilde I_{M_{\nu}}yx^{-1})\\
&=&y^{-1}\sigma(b_0\tilde I_{M_{\nu}}b_0^{-1})y\\
&=&I_M
\end{eqnarray*}
and similar translations for $N$ and $\overline N$. Hence $x$ is a $P$-fundamental alcove in $[b]$. For the other direction let $x$ be a $P$-fundamental alcove and let $w\in {}^MW^{M_{\nu}}$ with $w^{-1}Mw=M_{\nu}$. Then $w^{-1}(I\cap M)w=I\cap M_{\nu}$. A similar translation as above shows that $\tilde I=w^{-1}Iw$ and $b_0=\sigma^{-1}(w)^{-1}xw$ satisfy $\phi_{b_0}(P_{\nu})=P_{\nu}=\sigma(P_{\nu})$, hence $b_0\in\widetilde{W}_{M_{\nu}}$. Furthermore $\phi_{b_0}(I_{M_{\nu}})=I_{M_{\nu}}=\sigma(I_{M_{\nu}})$, whence $b_0I_{M_{\nu}}b_0^{-1}=I_{M_{\nu}}$ and $b_0\in \Omega_{M_{\nu}}$.
\end{proof}

\begin{lemma}\label{lemhelp3}
Let $M$ be the Levi factor containing $T$ of a standard parabolic subgroup $P$ and let $N$ be the unipotent radical of $P$. Let $I_1$ and $I_2$ be two Iwahori subgroups containing $I\cap M$ where $I$ is the standard Iwahori. Then there is a unique Iwahori subgroup $I'$ containing $I\cap M$, $I_1\cap N$, and $I_2\cap N$ and minimizing the intersection $I'\cap N$.
\end{lemma}

\begin{proof}
The Iwahori subgroups we are interested in correspond to alcoves in the apartment corresponding to $T$ in the Bruhat-Tits building of $G$. Note that an Iwahori subgroup $\tilde I$ containing $I\cap M$ satisfies $\tilde I\cap M =I\cap M$. The Iwahori subgroups $J$ containing $\tilde I\cap P$ correspond to the alcoves in the intersection of the half-spaces of the apartment corresponding to the conditions $J\cap U_{\alpha}\supseteq \tilde I\cap U_{\alpha}$ for each root $\alpha$ in $P$. We denote this subset of the apartment by $\mathcal{P}_{\tilde I}$. Note that $\mathcal{P}_{\tilde I}=\mathcal{P}_{\tilde I'}$ implies that $\tilde I\cap P=\tilde I'\cap P$ and thus $\tilde I=\tilde I'$. It is thus enough to show that $\mathcal{P}_{I_1}\cap\mathcal{P}_{I_2}=\mathcal{P}_{I'}$ for some $I'$. Note that our assumption $I_1\cap M=I_2\cap M$ implies that $\mathcal{P}_{I_1}\cap\mathcal{P}_{I_2}$ is non-empty. Let $\mathfrak{a}_1$ and $\mathfrak{a}_2$ denote the alcoves corresponding to $I_1$ and $I_2$. To prove the assertion above we use induction on the minimal distance in the building between $\mathfrak{a}_1$ and an alcove in $\mathcal{P}_{I_1}\cap\mathcal{P}_{I_2}$. If this distance is 0, then $\mathfrak{a}_1\in \mathcal{P}_{I_1}\cap\mathcal{P}_{I_2}$, hence $\mathcal{P}_{I_1}\cap\mathcal{P}_{I_2}=\mathcal{P}_{I_1}$. Assume now that $\mathfrak{a}_1\notin \mathcal{P}_{I_1}\cap\mathcal{P}_{I_2}$. Then there is an affine hyperplane bounding $\mathfrak{a}_1$ and with the property that $\mathfrak{a}_1$ and $\mathcal{P}_{I_1}\cap\mathcal{P}_{I_2}$ lie on different sides of this hyperplane. Let us denote this hyperplane by $H$. Let $\mathfrak{a}_1'$ be the alcove obtained from $\mathfrak{a}_1$ by reflection at $H$ and let $I_1'$ be the corresponding Iwahori subgroup. Then by definition the minimal distance of $\mathfrak{a}_1'$ to an element of $\mathcal{P}_{I_1}\cap\mathcal{P}_{I_2}$ is 1 less than the corresponding distance for $\mathfrak{a}_1$. Thus by induction it is enough to show that $\mathcal{P}_{I_1}\cap\mathcal{P}_{I_2}=\mathcal{P}_{I_1'}\cap\mathcal{P}_{I_2}$. For all affine hyperplanes $H'\neq H$, the two alcoves $\mathfrak{a}_1'$ and $\mathfrak{a}_1$ lie on the same side of $H'$. Let $S$ be the half-space bounded by $H$ and containing $\mathfrak{a}_1'$. Then $\mathcal{P}_{I'_1}= \mathcal{P}_{I_1}\cap S$. On the other hand we chose $H$ such that $\mathcal{P}_{I_1}\cap\mathcal{P}_{I_2}\subseteq S$. Hence $\mathcal{P}_{I_1}\cap\mathcal{P}_{I_2}=\mathcal{P}_{I_1'}\cap\mathcal{P}_{I_2}$.
\end{proof}
For a precise description of the sets $\mathcal{P}_I$ for the standard Iwahori compare \cite{GHKR2}, 3.

\begin{proof}[Proof of Theorem \ref{thmexfundalc}]
Let $\nu\in X_*(T)_{\mathbb{Q},\rm dom}$ be the dominant Newton point of $[b]$. Let $P_{\nu}$ be the associated parabolic subgroup and $P_{\nu}=M_{\nu}N_{\nu}$ the decomposition into the Levi factor containing $T$ and the unipotent radical. Recall that $\nu$, $P_{\nu}$, $M_{\nu}$ and $N_{\nu}$ are $\sigma$-invariant. Let $b_0\in \widetilde W_{M_{\nu}}$ be a \bsh element. By Lemma \ref{lemhelp2} it is enough to prove that there is an Iwahori subgroup $\tilde I$ of $LG$ with $\tilde{I}\cap M_{\nu}=I_{M_{\nu}}$ and such that $\phi_{b_0}(\tilde I\cap N)\subseteq \tilde I\cap N$. Let $r>0$ be such that $G$ is split over some unramified extension of $\mathcal{O}_F$ of degree $r$. In particular, $\sigma^r$ then acts trivially on the root system of $G$ and on $\widetilde W$. Let $c:=\sigma(b_0)\sigma^{2}(b_0)\dotsm \sigma^r (b_0)\in \widetilde W$. Applying the decomposition $\widetilde W=W\ltimes X_*(T)$ we obtain $c=w_c\mu_c$. Let $n_c$ be the order of $w_c$ in $W$. Replacing $r$ by $n_cr$ and using $\sigma(b_0)\sigma^{2}(b_0)\dotsm \sigma^{rn_c}b_0=c^{n_c}\in X_*(T)$ we may assume that we already have $c\in X_*(T)$. Note that as $b_0$ is $[b]$-short, we have $c=\nu_{r,M_{\nu}}=r\nu$ as elements of $X_*(T)_{\mathbb{Q}}$. In particular, $c$ is dominant and central in $M_{\nu}$. Using Section \ref{sec25} and the fact that $\sigma^r$ acts trivially on $\widetilde W$ we obtain that $\sigma^r(cI_{M_{\nu}}c^{-1})=cI_{M_{\nu}}c^{-1}=I_{M_{\nu}}$ and $\sigma^r(cI_{N_{\nu}}c^{-1})= cI_{N_{\nu}}c^{-1}\subseteq I_{N_{\nu}}$. Hence $c$ itself is a $P_{\nu}$-fundamental alcove for the $\sigma^r$-conjugacy class of $c$ in $G$. Let $\tilde I$ with $\tilde I\cap M_{\nu}=I_{M_{\nu}}$ be unique Iwahori subgroup of $LG$ such that $\tilde I\cap N_{\nu}$ is minimal containing $I_{N_{\nu}}, \phi_{b_0}(I_{N_{\nu}}),\dotsc, \phi_{b_0}^{r-1} (I_{N_{\nu}})$, cf. Lemma \ref{lemhelp3}. Then $\phi_{b_0}(\tilde I)$ is again an Iwahori subgroup. It satisfies $\phi_{b_0}(\tilde I)\cap M_{\nu}=\phi_{b_0}(\tilde I\cap M_{\nu})=I_{M_{\nu}}$ and the analogous minimality property for $\phi_{b_0}(I_{N_{\nu}}),\dotsc, \phi_{b_0}^{r} (I_{N_{\nu}})$. We have $\phi_{b_0}^r(I_{N_{\nu}})=\sigma^r(cI_{N_{\nu}}c^{-1})\subseteq I_{N_{\nu}}$. Thus $\phi_{b_0}(\tilde I\cap N_{\nu})\subseteq \tilde I\cap N_{\nu}$, and $\tilde I$ is as claimed.
\end{proof}

\begin{remark}\label{newremark'}
Denote for the moment by $B(G)_{k((t))}$ the set of $\sigma$-conjugacy classes in $G(k((t)))$ and by $B(G)_{W(k)[1/p]}$ the corresponding set in $G(W(k)[1/p])$. Kottwitz's classification maps both sets injectively to $X_*(T)_{\mathbb{Q}}\times \pi_1(G)_{\Gamma}$. Note that $X_*(T)_{\mathbb{Q}}\times \pi_1(G)_{\Gamma}$ only depends on $G_{\mathbb{F}_q}$ but not on $k$ or on the choice of the arithmetic or the function field case. Furthermore the images of $B(G)_{k((t))}$ and $B(G)_{W(k)[1/p]}$ in $X_*(T)_{\mathbb{Q}}\times \pi_1(G)_{\Gamma}$ can also be described in terms of $G_{\mathbb{F}_q}$, and are independent of the choice of $L$. In particular we obtain canonical bijections $B(G)_{k((t))}\cong B(G)_{W(k)[1/p]}$ and $B(G)_{k((t))}\cong B(G)_{k'((t))}$ for all algebraically closed fields $k'$ of characteristic $p$. From now on we use these bijections to identify the sets of $\sigma$-conjugacy classes and write again $B(G)$ for all of them.
\end{remark}
The main reason to introduce fundamental alcoves in this paper is the following proposition which yields a direct comparison between non-emptiness of intersections of $\sigma$-conjugacy classes and Iwahori double cosets in the function field case and in the arithmetic case.

\begin{prop}\label{lemsuperset}
\begin{enumerate}
\item Let $L=k((t))$ or $\Quot (W(k))$. Let $[b]\in B(G)$ and let $x_b$ be a $P$-fundamental alcove contained in $[b]$. Then $$\{x \in \widetilde W\mid I x I \cap [b] \neq \emptyset\} = \{x \in \widetilde W\mid x\in  I y^{-1} I x_b I\sigma( y) I\text{ for some } y\in \widetilde{W}\}.$$
\item Let $x\in \widetilde{W}$. Then a $\sigma$-conjugacy class in $G(W(k)[1/p])$ contains an element of $IxI$ (for $I$ defined with respect to $W(k)$) if and only if the corresponding $\sigma$-conjugacy class in $G(k((t)))$ contains an element of $IxI$ (where $I$ is now a subgroup of $LG(k)$).
\end{enumerate}
\end{prop}
For split groups the first assertion is \cite{GHKR2}, Proposition 13.3.1. Our statement follows using the same proof. As it is very short we repeat it for the reader's convenience. 
\begin{proof}
Let $g\in IxI\cap [b]$. Then there is an $h\in LG$ with $h^{-1}x_b\sigma(h)=g$. Let $y\in \widetilde W$ with $h\in IyI$. Then $x\in IxI=IgI\subset Iy^{-1}Ix_b I\sigma(y)I$. For the other direction let $x\in I y^{-1} I x_b I\sigma( y) I$ for some $y\in \widetilde{W}$. Then $IxI\cap y^{-1}Ix_bI\sigma(y)\neq\emptyset$. Recall that every element of $Ix_bI$ is of the form $i^{-1}x_b\sigma(i)$ for some $i\in I$ (Lemma \ref{remoortmain}). Thus $IxI$ contains an element of the form $y^{-1}i^{-1}x_b\sigma(iy)\in [b]$.

From (1) together with Theorem \ref{thmexfundalc} we see that both conditions in (2) can be translated into the same condition in terms of the combinatorics of $\widetilde{W}$ which is independent of the choice of $L$. Thus (2) follows.
\end{proof}

In particular, we can now easily deduce Theorem \ref{corsuperset}.
\begin{proof}[Proof of Theorem \ref{corsuperset}]
This follows from Proposition \ref{lemsuperset} (2) together with Theorem \ref{propdef} (2).
\end{proof}
We finish our discussion of fundamental alcoves by a comparison to \sh elements.
\begin{lemma}\label{lemfundshort}
Let $G$ be split. Then every $P$-fundamental alcove in a given $\sigma$-conjugacy class $[b]$ is $W$-conjugate to the unique \bsh element.
\end{lemma}
Note that for split groups, $W$-conjugation coincides with $W$-$\sigma$-conjugation.
\begin{proof}
Let $b_0$ be $P$-fundamental and contained in $[b]$. We write $P=MN$ for the unipotent radical $N$ of $P$ and the Levi factor $M$ containing $T$. Let $x$ be the \bsh element and let $M_{\nu}$ be the centralizer of the dominant Newton point $\nu$ of $b$. By definition $x\in \widetilde{W}$ is an element of length 0 in $\widetilde{W}_{M_{\nu}}$. By Lemma \ref{lemnew} we may assume that $M$ is equal to the centralizer of the $M$-dominant Newton point of $b_0$. As $[b_0]=[b]$, their Newton points coincide and $\kappa_G(b_0)=\kappa_G(b)$. As in the proof of Lemma \ref{lemshort} the Newton point of $b_0$ together with $\kappa_G(b_0)$ determines $\kappa_{M}(b_0)$. Hence there is a $w\in W$ that conjugates $M$ to $M_{\nu}$ and also $\kappa_{M}(b_0)$ to the element $\psi_b\in\pi_1(M_{\nu})$ defined in the proof of Lemma \ref{lemshort}. Choosing $w$ of minimal length in its coset $W_{M}wW_{M_{\nu}}$ it also conjugates $I_M$ to $I_{M_{\nu}}$. Now $w^{-1}b_0w$ and $x$ are in $\widetilde{W}_{M_{\nu}}$ and both have length 0 as elements of $\widetilde{W}_{M_{\nu}}$. Thus they are in the subgroup $\Omega_{M_{\nu}}$. As $G$ is split, we have that $\kappa_{M_{\nu}}:\Omega_{M_{\nu}}\rightarrow \pi_1(M_{\nu})$ is an isomorphism. As the images of $x$ and $w^{-1}b_0w$ under $\kappa_{M_{\nu}}$ coincide, the elements have to be equal. Hence $b_0$ is $W$-conjugate to the \bsh element $x$.
\end{proof}

\section{Applications}\label{secpolpdiv}
In this section we consider the case $F=\mathbb{Q}_p$. We review some of the theory of Ekedahl-Oort strata and relate it to our notion of truncations of level 1. We concentrate on the example of the moduli space $\mathcal{A}_g$ of principally polarized abelian varieties of dimension $g$, and briefly indicate possible generalizations to other Shimura varieties of PEL type. For more details on this general theory of Ekedahl-Oort strata we refer to \cite{VW}.

The Ekedahl-Oort stratification is the stratification of $\mathcal{A}_g$ according to the $p$-torsion $(A,\lambda)[p]$ of the principally polarized abelian varieties $(A,\lambda)$ associated with the points of $\mathcal{A}_g$. It was first defined and studied by Oort in \cite{EO}. Oort classifies the $p$-torsion $(A,\lambda)[p]$ by a finite combinatorial invariant, so-called elementary sequences. A second description of the Ekedahl-Oort invariant (and more generally of $G(k)$-orbits on a certain variety associated with a reductive group $G$ over $\mathbb{Z}_p$ together with a fixed Levi subgroup that is also defined over $\mathbb{Z}_p$) has been given by Moonen and Wedhorn in \cite{MoonenWedhorn}. They use a description by so-called $F$-zips and identify the index set for the Ekedahl-Oort stratification of $\mathcal{A}_g$ with ${}^{\mu}W$ where $\mu$ is the minuscule dominant element given by the Shimura datum and where $W$ is the Weyl group of $GSp_{2g}$. Another related theory is Vasiu's classification of so-called Shimura $F$-crystals in \cite{Vasiu}, Main Theorem C.

In the language of truncations of level $1$ the Ekedahl-Oort invariant on $\mathcal A_g$ can be studied as follows. From an element of $\mathcal{A}_g(k)$ we obtain a polarized $p$-divisible group $(A,\lambda)[p^{\infty}]$. The polarization equips its Dieudonn\'e module with a symplectic form $\langle \cdot,\cdot\rangle$. In the same way as in Section \ref{compeo} we trivialize its Dieudonn\'e module and obtain that the Frobenius is given by an element $b$ of $GSp_{2g}(W(k)[1/p])$, well-defined up to $\sigma$-conjugation with $GSp_{2g}(W(k))=K(k)$. It satisfies that $b\in K(k)\mu(p)K(k)$ where $\mu$ is as above, i.e.~$\mu(p)$ is the diagonal matrix with entries $p$ and $1$, each with multiplicity $g$. The Ekedahl-Oort stratification is then nothing but the stratification that one obtains by considering the truncations of level $1$ of the elements $b$. The index set for truncations of level 1 of elements of $K\mu(p)K$ is equal to ${}^{\mu}W$ (Theorem \ref{propdef}(1)). This identification coincides with the one used in the classification by Moonen and Wedhorn.

For $w\in {}^{\mu}W$ let $S_w$ be the reduced subscheme of the reduction of $\mathcal{A}_g$ given by the condition that $(A,\lambda)[p]$ has Ekedahl-Oort invariant $w$. Oort proves that each stratum $S_w$ is locally closed, and the closure $\overline{S_w}$ is a union of strata. The set of strata that are contained in $\overline{S_w}$ is determined in \cite{Wedhorn} together with (6.4) of loc.~cit. It is given by the same formula as the closure relations between the corresponding strata $S_{w,\mu}$ in the loop group of $G=GSp_{2g}$ (that we compute in Corollary \ref{corclosure}).

Recall that in Theorem \ref{propdef}(2) we established a comparison between the stratification by truncations of level 1 and the subdivision of $LG$ into Iwahori double cosets. Relations between the Ekedahl-Oort stratification and the subdivision into Iwahori-double cosets are also used in the theory of moduli spaces of abelian varieties, see for example \cite{EkedahlvdGeer}, Corollary 8.4 (iii) or \cite{GY2}, 9.\\

We now compare Oort's minimal $p$-divisible groups (see \cite{Oort2}) to our notion of \sh elements. Let $X$ be a $p$-divisible group over an algebraically closed field $k$ and let $(\mathbf M,F)$ be its Dieudonn\'e module. Let $\mathbf N= \mathbf M\otimes_{W(k)}\Quot(W(k))$. By definition there is a unique isomorphism class of minimal $p$-divisible groups in each isogeny class of $p$-divisible groups (see \cite{Oort2}). Explicitly, if $X$ is minimal, its Dieudonn\'e module is isomorphic to a Dieudonn\'e module of the following form. There is a decomposition of the rational Dieudonn\'e module into simple summands $\mathbf N=\bigoplus_{i=1}^l \mathbf N_l$ such that $\mathbf M=\bigoplus_{i=1}^l \mathbf M\cap \mathbf N_i$. Let $\lambda_i=n_i/h_i$ with $(n_i,h_i)=1$ be the slope of $\mathbf N_i$. Then $\mathbf M\cap \mathbf N_i$ has a basis $e^i_1,\dotsc,e^i_{h_i}$ such that $F(e^i_j)=e^i_{j+n_i}$. Here we use the notation $e^i_{j+h_i}=pe^i_j$. Equivalently, $X$ is minimal if the endomorphisms of $(\mathbf M,F)$ are a maximal order in the endomorphisms of $(\mathbf N,F)$.

Let now $f^i_j=e^i_{h_i+1-j}$. Let $h=\dim \mathbf N$. One easily checks that if we write $F=b\sigma$ for $b\in GL_{h}(L)$ with respect to the basis $f^1_1,\dotsc,f^1_{h_1},f^2_1,\dotsc$, then $b$ is contained in the Levi subgroup $M$ given by the decomposition $\mathbf N=\bigoplus_{i=1}^l\mathbf N_i$. Furthermore, if $\mu$ denotes the $M$-dominant Hodge polygon of $b$ (with respect to the choice of the upper triangular matrices as Borel subgroup), then $\mu\in\{0,1\}^h$ is minuscule and $b=\tau_{\mu,M}$ satisfies $bI_Mb^{-1}=I_M$. Hence a $p$-divisible group is minimal if and only if the $K$-$\sigma$-conjugacy class of the element determining the Frobenius on the Dieudonn\'e module contains a \sh element, or equivalently (by Lemma \ref{lemfundshort}) a $P$-fundamental alcove.
 
\begin{kor}\label{coroort}
Let $X$ be a minimal $p$-divisible group over $k$ and let $Y$ be a $p$-divisible group with $X[p]\cong Y[p]$. Then $X\cong Y$. An analogous assertion holds for polarized $p$-divisible groups. 
\end{kor}
This reproves the main theorem of \cite{Oort2}.
\begin{proof}
We use Dieudonn\'e theory and trivialize the Dieudonn\'e modules of $X$ and $Y$ to reformulate the assertion. Let $G=GL_h$ resp. $GSp_{h}$ where $h$ is the height of $X$. Let $b_X\in G(W(k)[1/p])$ be the element describing the Frobenius on the Dieudonn\'e module of $X$. As $X$ is minimal we can choose the trivialization in such a way that $b_X$ is a $P$-fundamental alcove for some $P$. Let $b_Y\in G(W(k)[1/p])$ be the element describing the Frobenius on the Dieudonn\'e module of $Y$. As $X[p]\cong Y[p]$ we can choose the trivialization in such a way that $b_Y\in K_1 b_XK_1$. By Lemma \ref{remoortmain} $b_X$ and $b_Y$ are $I$-$\sigma$-conjugate to each other. In particular, the are $K$-$\sigma$-conjugate which implies that $X\cong Y$.
\end{proof}

It would be interesting to construct a generalization of minimal $p$-divisible groups for all good reductions of PEL Shimura varieties. In particular, one would be interested in a representative of a given isogeny class of $p$-divisible groups with endomorphisms and polarization satisfying the analogue of Corollary \ref{coroort}. Although we constructed $P$-fundamental alcoves for all $\sigma$-conjugacy classes of elements of $G(L)$ for all $G$, our theory does not imply the existence of such minimal $p$-divisible groups with extra structure for non-split $G$. The reason is that we did not study whether there exist $P$-fundamental alcoves in a given $\sigma$-conjugacy class which in addition lie in the prescribed $K$-double coset given by $\mu$. A weaker generalization of the notion of minimality would be to call a $p$-divisible group with PEL structure minimal if the $K$-$\sigma$-conjugacy class of the element determining the Frobenius on the Dieudonn\'e module contains a \sh element. Our theory implies the existence of such elements in each isogeny class, compare the discussion after Corollary \ref{thmmain'}.\\

One interesting open question about Ekedahl-Oort strata is to determine which Newton polygons occur in a given Ekedahl-Oort stratum. Our theory for loop groups (in particular Theorem \ref{thmmain}) together with the comparison results of the preceding section yield the following necessary condition.

\begin{kor}\label{thmmain'}
Let $x$ be a $k$-valued point of $S_w$ for some $w$. Let $x_0\in \mathcal{A}_g(k)$ be a point corresponding to the minimal $p$-divisible group in the isogeny class corresponding to $x$. Then $x_0\in \overline{S_w}$.
\end{kor}
\begin{proof}
Let $\mu$ be the Hodge vector associated with $\mathcal{A}_g$. Let $[b]\in B(G)$ be the class corresponding to the isogeny class of the $p$-divisible group corresponding to $x$. By Theorem \ref{corsuperset} the corresponding class $[b]$ in $LG$ intersects the truncation stratum $S_{w,\mu}\subseteq LG$. Let $b_0\in \widetilde W$ be a \bsh element. Then the representative of $b_0$ in $GSp_{2g}(W(k)[1/p])$ describes the Dieudonn\'e module of the minimal $p$-divisible group in $[b]$. Let $(w_0,\mu)$ be its truncation type. From Theorem \ref{thmmain} we obtain that $S_{w_0,\mu}\subseteq \overline {S_{w,\mu}}$ in $LG$. Recall that the closure relations between the strata $S_{w}$ are known to coincide with those between the corresponding strata $S_{w,\mu}$ in the loop group of $G=GSp_{2g}$. Thus the corollary follows.
\end{proof}

For the Siegel moduli space $\mathcal{A}_g$ this has been conjectured by Oort \cite{Oort1}, Conjecture 6.9 and has been shown previously by Harashita in \cite{H1}, \cite{H2}, \cite{H3} using different methods. While this article was being finished, Harashita published a preprint \cite{H4} in which he proves an analog of Corollary \ref{thmmain'} for some catalog of $p$-divisible groups in the non-polarized case (without endomorphisms). Our approach to prove Corollary \ref{thmmain'} also leads to variants without polarization, and/or with endomorphisms: Let $S_w$ denote the truncation strata in a moduli space of abelian varieties associated with a PEL Shimura variety with good reduction at $p$. The same proof as above then shows that $x\in S_w(k)$ for some $w\in {}^{\mu}W$ implies that there is an element $x_0$ whose $p$-divisible group (with extra structure) is isogenous to the one corresponding to $x$, such that the associated element $b_{x_0}\in G(W(k)[1/p])$ is short, and such that $x_0\in\overline{S_w}$. This element is in general (for non-split $G$) not uniquely defined by the isogeny class (compare Lemma \ref{lemshort}). For more details we refer to \cite{VW}.

For the moduli space $\mathcal{A}_g$ of principally polarized abelian varieties of dimension $g$ in characteristic $p>2$ we know by \cite{EkedahlvdGeer}, Theorem 11.5 that each Ekedahl-Oort stratum which is not contained in the supersingular locus is irreducible. In particular, there is a unique generic Newton polygon in each Ekedahl-Oort stratum $S_w$ of $\mathcal{A}_g$. Then in the same way as for the loop group we can use the above result to determine this Newton polygon.
\begin{kor}
Let $\nu$ be the generic Newton polygon in $S_w\subseteq \mathcal{A}_g$ for some $w\in {}^{\mu}W$. Then $\nu$ is the maximal element in the set of Newton polygons of short elements $x$ such that $x\in \overline{S_{w}}$.
\end{kor}

\end{document}